\documentclass[11pt]{amsart}

\usepackage{amsfonts, amsmath, amscd}
\usepackage[psamsfonts]{amssymb}

\usepackage{amssymb}

\usepackage{pb-diagram}

\usepackage[all,cmtip]{xy}

\usepackage[usenames]{color}

\newtheorem{theorem}{Theorem}[section]
\newtheorem{lemma}[theorem]{Lemma}
\newtheorem{corollary}[theorem]{Corollary}
\newtheorem{question}[theorem]{Question}
\newtheorem{remark}[theorem]{Remark}
\newtheorem{proposition}[theorem]{Proposition}
\newtheorem{definition}[theorem]{Definition}
\newtheorem{example}[theorem]{Example}
\newtheorem{fact}[theorem]{Fact}

\newtheorem{problem}[theorem]{Problem}

\numberwithin{equation}{section}

\newcommand{\CC}{C_k}
\newcommand{\NN}{\mathbb{N}}
\newcommand{\ZZ}{\mathbb{Z}}
\newcommand{\CP}{\mathbf{cp}}

\newcommand{\GG}{\mathfrak{G}}

\newcommand{\w}{\omega}

\newcommand{\KK}{\mathcal{K}}

\newcommand{\VV}{\mathbb{V}}

\newcommand{\IR}{\mathbb{R}}

\newcommand{\Ss}{\mathbb{S}}
\newcommand{\ff}{\mathbb{F}}
\newcommand{\Ff}{\mathfrak{F}}

\newcommand{\Pp}{\mathfrak{P}}
\newcommand{\e}{\varepsilon}

\renewcommand{\phi}{\varphi}

\newcommand{\supp}{\mathrm{supp}}

\newcommand{\spn}{\mathrm{span}}

\newcommand{\TT}{\mathbb{T}}

\newcommand{\Bo}{\mathsf{Bo}}

\input xy
\xyoption{all}

\title[Maximally almost periodic groups and respecting properties]{Maximally almost periodic groups \\ and respecting properties}

\author[S.~Gabriyelyan]{Saak Gabriyelyan}
\address{Department of Mathematics, Ben-Gurion University of the
Negev, Beer-Sheva, P.O. 653, Israel}
\email{saak@math.bgu.ac.il}



\begin{document}

\begin{abstract}
For a Tychonoff space $X$, denote by $\Pp$ the family of topological properties $\mathcal{P}$ of being  a convergent sequence or being a compact, sequentially compact, countably compact, pseudocompact and functionally bounded subset of $X$, respectively. A maximally almost periodic $(MAP$) group $G$ respects $\mathcal{P}$ if $\mathcal{P}(G)=\mathcal{P}(G^+)$, where $G^+$ is  the group $G$ endowed with the Bohr topology. We study relations between different respecting properties from $\Pp$ and show that the respecting convergent sequences (=the Schur property) is the weakest one among the properties of $\Pp$.
We characterize respecting properties from $\Pp$ in wide classes of $MAP$ topological groups including the class of metrizable $MAP$ abelian groups. Every real locally convex space (lcs) is a quotient space of an lcs with the Schur property, and every locally quasi-convex (lqc) abelian group is a quotient group of an lqc abelian group with the Schur property. It is shown that a reflexive group $G$ has the Schur property or respects compactness  iff its dual group $G^\wedge$ is $c_0$-barrelled or $g$-barrelled, respectively. We prove that an lqc abelian $k_\w$-group respects all properties $\mathcal{P}\in\Pp$. As an application of the obtained results we show that (1) the space $\CC(X)$ is a reflexive group for every separable metrizable space $X$, and  (2) a reflexive abelian group of finite exponent is a Mackey group.
\end{abstract}

\maketitle


\section{Introduction}

Let $X$ be a Tychonoff space. If $\mathcal{P}$ is a topological property, we denote by $\mathcal{P}(X)$ the set of all subspaces of $X$ with $\mathcal{P}$.
Denote by $\mathcal{S}$, $\mathcal{C}$, $\mathcal{SC}$, $\mathcal{CC}$, $\mathcal{PC}$ or $\mathcal{FB}$ the property of being  a convergent sequence or being a compact, sequentially compact, countably compact, pseudocompact and functionally bounded subset of $X$, respectively. In what follows we consider the following families of compact-type topological properties
\[
\mathfrak{P}_0 := \{ \mathcal{S}, \mathcal{C}, \mathcal{SC}, \mathcal{CC}, \mathcal{PC}\} \quad \mbox{ and } \quad \mathfrak{P} :=\mathfrak{P}_0 \cup \{\mathcal{FB}\}.
\]

Let $G$ be a maximally almost periodic ($MAP$) topological group $G$ (for all relevant definitions, see Section \ref{sec:prelim}). We denote by $G^+$ the group $G$ endowed with the Bohr topology. Following \cite{ReT3}, a $MAP$ group $G$ {\em respects} a topological property $\mathcal{P}$  if $\mathcal{P}(G)=\mathcal{P}(G^+)$.

The famous Glicksberg theorem \cite{Gli} states that every locally compact abelian ($LCA$) group respects compactness. If a $MAP$ group $G$ respects compactness we shall say also that $G$ has the {\em  Glicksberg property}. Trigos-Arrieta \cite{TriAr91, Tri91} proved that  countable compactness, pseudocompactness  and functional boundedness are respected by $LCA$ groups. Banaszczyk and Mart\'{\i}n-Peinador \cite{BaMP2} generalized these results to  all nuclear groups. Nuclear groups were introduced and thoroughly studied by Banaszczyk in \cite{Ban}. The concept of  Schwartz topological abelian groups is appeared in \cite{ACDT}. This notion generalizes the well-known notion of a Schwartz locally convex space.  All nuclear groups are Schwartz groups \cite{ACDT}. Au{\ss}enhofer \cite{Aus2} proved that every locally quasi-convex Schwartz group respects compactness.
For a general and simple approach to the theory of properties  respected by $MAP$ topological groups see \cite{Gab-Top-Nul}.

Let $(E,\tau)$ be a  locally convex space (lcs for short), $E'$ the dual space of $E$ and let $\tau_w=\sigma(E,E')$ be the weak topology on $E$. Set $E_w:=(E,\tau_w)$. An lcs $E$ is said to have the {\em  Schur property} if $E$ and $E_w$ have the same convergent sequences, i.e., $\mathcal{S}(E)=\mathcal{S}(E_w)$. Considering $E$ as an additive topological group one can define $E^+$.
If $E$ is a real lcs, it is proved in \cite{ReT} that $E_w$ and $E^+$ have the same compact sets and hence the same convergent sequences. By this reason we shall use of the terminology ``the Schur property'' and ``the  Glicksberg property'' for locally convex spaces and $MAP$ groups simultaneously.
It is easy to see that if a $MAP$ group $G$ has the  Glicksberg property, then it has also the Schur property. In general the Schur property does not imply the Glicksberg property, see \cite[Example~6 (p.~267)]{Wil} and \cite[Example~19.19]{Dom}, or \cite[Proposition~3.5]{Gabr-free-resp} for a more general assertion. However, it is a classical result that a Banach space $E$ has the Schur property if and only if $E$ has the  Glicksberg property.

The aforementioned results motivate us to consider the following two problems. The first problem concerns finding relationships between the properties $\mathcal{P}\in \Pp$. In Section \ref{sec:2} we show that the Schur property is equivalent to the respecting sequential compactness and is the weakest one among the properties of $\Pp_0$, see Proposition \ref{p:Schur=sequential-comp}. Under the additional assumption that $G^+$ is a $\mu$-space, we prove in Theorem \ref{t:Glicksberg-respecting} that the Glicksberg property implies all other  properties $\mathcal{P}\in \Pp$. In Proposition \ref{p:Schur-E-S} and  Theorem \ref{t:Schur-Bohr-angelic} we consider some natural classes of $MAP$ groups in which the Schur property implies the Glicksberg property and the respecting countable compactness.
We show also that the Schur property and the Glicksberg property have a natural categorical characterization, see Proposition \ref{p:Schur-categorical}.


The second natural problem is the following: Characterize respecting properties in concrete classes of $MAP$ groups. As we mentioned above, every locally quasi-convex Schwartz group has the Glicksberg property. However, there are even metrizable reflexive abelian  groups which are not Schwartz groups, see \cite{Gab-Top-Nul}.
%
In  \cite{HGM} Hern\'{a}ndez, Galindo and Macario proved that a reflexive metrizable abelian group $G$ has the Glicksberg property if and only if every non-precompact subset $A$ of $G$ has an infinite subset $B$ which is discrete and $C^\ast$-embedded in the Bohr compactification $bG$ of $G$. This result was generalized by Hern\'{a}ndez and Macario in \cite{HM} who  showed that a complete abelian $g$-group $G$ has the Glicksberg property if and only if $G$  respects functional boundedness if and only if every non-precompact subset $A$ of $G$ has an infinite subset $B$ which is discrete and $C$-embedded in $G^+$. Below we generalize this result, see Theorem \ref{t:Glicksberg-respecting}. Let us recall (see \cite{HM}) that every complete $g$-group $G$ is semi-reflexive and $G^+$ is a $\mu$-space. However, Au{\ss}enhofer found in \cite{Aus} 
an example of a metrizable complete locally quasi-convex abelian group which is not semi-reflexive. Thus the above results do not give a characterization of the Glicksberg property  even in the class of complete metrizable abelian groups.

We showed in \cite{Gab-Top-Nul} that if a complete $MAP$ group respects functional boundedness, then $G^+$ must be  a $\mu$-space. Therefore to obtain respecting properties for a $MAP$ group $G$ we should assume that $G$ and $G^+$ satisfy  some completeness type properties.
We  say that a topological space $X$ is a {\em countably $\mu$-space} if every countable functionally bounded  subset of $X$ has compact closure. Clearly, every $\mu$-space is a countably $\mu$-space, but the converse is not true  in general (see Example \ref{exa:seq-compact-non-compact} below). We shall say that a $MAP$ group $G$ is {\em Bohr angelic} if $G^+$ is angelic. In Theorem \ref{t:Schur-Bohr-angelic} we characterize the Glicksberg property by means of the Bohr topology in the class of Bohr angelic groups $G$  which are countably $\mu$-spaces. The class of Bohr angelic $MAP$ groups  is sufficiently rich since it contains all  $MAP$ abelian groups with a $\GG$-base, see Proposition \ref{p:web-compact-Bohr-angelic}.
The class of topological groups with a $\GG$-base is introduced in \cite{GKL}, it contains all metrizable groups and  is closed under taking completions, quotients, countable products and countable direct sums. Using  Theorem \ref{t:Schur-Bohr-angelic} we prove the following general result.
\begin{theorem}  \label{t:G-base-Schur}
Let $G$ be a $MAP$ abelian group  with a $\GG$-base. If $G$ is a countably $\mu$-space, then the following assertions are equivalent:
\begin{enumerate}
\item[{\rm (i)}] $G$ has the Schur property;
\item[{\rm (ii)}] $G$ has the Glicksberg property;
\item[{\rm (iii)}] $G$ respects  sequential compactness;
\item[{\rm (iv)}] $G$ respects  countable compactness;
\item[{\rm (v)}] every non-functionally bounded subset of $G$ has an infinite subset which is closed and discrete in $G^+$.
\end{enumerate}
If, in addition, $G^+$ is a $\mu$-space, then (i)-(v) are equivalent to the following assertions:
\begin{enumerate}
\item[{\rm (vi)}] $G$ respects  pseudocompactness;
\item[{\rm (vii)}] $G$ respects  functional boundedness.
\end{enumerate}
\end{theorem}
In particular, since every metrizable space is a $\mu$-space and hence a countably $\mu$-space, Theorem \ref{t:G-base-Schur} gives a characterization of the Glicksberg property in the class of metrizable $MAP$ abelian groups without the restrictive assumption of being a $g$-group as in \cite{HM}. 

An important class of locally convex spaces (lcs for short) is the class of free locally convex spaces $L(X)$ over Tychonoff spaces $X$.  In Section \ref{sec:Free-Schur}
we prove the following result.
\begin{theorem} \label{t:L(X)-Schur}
Let $X$ be a Tychonoff space. Then the free locally convex space $L(X)$ over $X$ respects all properties $\mathcal{P}\in\Pp_0$. If $L(X)$ is complete, then $L(X)$  respects all properties $\mathcal{P}\in\Pp$.
\end{theorem}

It is well known that  every Banach space $E$ is a quotient space of a Banach space with the Schur property (more precisely, $E$ is a quotient space of the space $\ell_1(\Gamma)$ for some set $\Gamma$). We generalize this result to all real locally convex spaces and all abelian Hausdorff topological groups.

\begin{corollary} \label{c:L(X)-quotients}
Every real locally convex space $E$ is a quotient space of an lcs which weakly respects all properties $\mathcal{P}\in\Pp_0$.
\end{corollary}

\begin{corollary} \label{c:A(X)-quotients}
Every Hausdorff abelian topological group $G$ is a quotient group of a locally quasi-convex abelian group which respects all properties $\mathcal{P}\in\Pp_0$.
\end{corollary}

As an application we prove in Section \ref{sec:Free-Schur} that the space $\CC(X)$ is a reflexive group for every separable metrizable space $X$, see Proposition \ref{p:reflexive-C(X,G)}.

In Section \ref{sec:3} we obtain  dual characterizations of the Schur property and the Glicksberg property in the class of reflexive abelian groups by showing that a reflexive group $G$ has the Schur property or the Glicksberg property  if and only if its dual group $G^\wedge$ is $c_0$-barrelled or $g$-barrelled, respectively (see Proposition \ref{p:Mackey-g-barrelled}). Another important class of topological groups is the class of abelian $k_\w$-groups. This class contains all dual groups of metrizable abelian groups, see \cite{Aus,Cha}.
In Theorem \ref{t:k-omega-Schur} we show that every locally quasi-convex $k_\w$-group $G$  respects all properties $\mathcal{P}\in \Pp$ (note that the condition of being a locally quasi-convex group cannot be omitted, see  Example \ref{exa:k-omega-not-Schur}).

In the last section we define some Glicksberg type properties and apply the obtained results to show that a reflexive abelian group of finite exponent is a Mackey group, see Theorem \ref{t:Mackey-finite-exp}. 


\section{Preliminary results} \label{sec:prelim}


Denote by $\mathbf{TG}$  ($\mathbf{TAG}$) the category of all Hausdorff (respectively, abelian) topological groups and continuous homomorphisms. A compact group $bX$ is called the {\em Bohr compactification} of $(X,\tau)\in \mathbf{TG}$ if there exists a continuous homomorphism $i$ from $X$ onto a dense subgroup of $bX$ such that the pair $(bX,i)$ satisfies the following {\em universal property}: If $p:X\to C$ is a continuous homomorphism into a compact group $C$, then there exists a continuous homomorphism $j^p: bX \to C$ such that $p=j^p\circ i$. Following von Neumann \cite{Neu},  the group $X$ is called {\em maximally almost periodic} ($MAP$) if the group $X^+$ is Hausdorff, where $X^+:=(X,\tau^+)$ is the group $X$ endowed with the Bohr topology $\tau^+$ induced from $bX$. The family $\mathbf{MAP}$ ($\mathbf{MAPA}$) of all $MAP$ (respectively, $MAP$ abelian) topological groups is a subcategory of $\mathbf{TG}$ (respectively, $\mathbf{TAG}$).
Every  irreducible representation of a (pre)compact group $X$ is finite-dimensional,  see  \cite[22.13]{HR1}. For an $X\in \mathbf{MAP}$, we denote by $\widehat{X}$  the set of all  (equivalence classes of) finite-dimensional irreducible representations of $X$. The {\em Bohr functor} $\mathfrak{B}$ on $\mathbf{MAP}$ is defined by  $\mathfrak{B}(X):=X^+$ for a $MAP$ group $X$ and $\mathfrak{B}(T)=T$ if $T:X\to Y$ is a continuous homomorphism. Denote by $\mathbf{PCom}$ the class of all precompact  groups.
If  a $MAP$ group $(G,\tau)$ is abelian, then every $\pi\in \widehat{G}$ is one-dimensional (indeed, since the unitary group of a finite-dimensional Hilbert space is compact, by the universal property, $\pi$ can be extened to $\widehat{\pi}\in \widehat{bG}$ and \cite[22.17]{HR1} applies), so $\widehat{G}$ coincides with the group of all continuous characters of $G$ denoted also by $\widehat{G}$. In  this case   $\tau^+=\sigma(G,\widehat{G})$, where $\sigma(G,\widehat{G})$ is the smallest group topology on $G$ for which the elements of $\widehat{G}$ are continuous.

Denote by $\mathbb{S}$ the unit circle group and set $\Ss_+ :=\{z\in  \Ss:\ {\rm Re}(z)\geq 0\}$.
Let $G$ be an abelian topological group.   A character $\chi\in \widehat{G}$  is a continuous homomorphism from $G$ into $\mathbb{S}$. A subgroup $H$ of $G$ is called {\em dually embedded} if every continuous character of $H$ can be extended to a continuous character of $G$.
A subset $A$ of $G$ is called {\em quasi-convex} if for every $g\in G\setminus A$ there exists   $\chi\in \widehat{G}$ such that $\chi(g)\notin \Ss_+$ and $\chi(A)\subseteq \Ss_+$.
If $A\subseteq G$ and $B\subseteq \widehat{G}$ set
\[
A^\triangleright :=\{ \chi\in \widehat{G}: \chi(A)\subseteq \Ss_+\}, \quad B^\triangleleft:=\{ g\in G: \chi(g)\in\Ss_+ \; \forall \chi\in B\}.
\]
Then $A$ is quasi-convex if and only if $A^{\triangleright\triangleleft}=A$. The set $\mathrm{qc}(A):=\bigcap_{\chi\in A^{\triangleright}} \chi^{-1} (\Ss)$ is called the {\em quasi-convex hull } of $A$.
An abelian topological group $(G,\tau)$ is called {\em locally quasi-convex} if it admits a neighborhood base at the neutral element $0$ consisting of quasi-convex sets. If $G$ is $MAP$, then the sets $\mathrm{qc}(U)$, where $U$ is a neighborhood of zero in $G$, form a neighborhood base of a locally quasi-convex group topology $\tau_{\mathrm{qc}}$, we set $G_{qc} :=(G,\tau_{\mathrm{qc}})$.  The class $\mathbf{LQC}$ of all abelian locally quasi-convex groups is one of the most important subclasses of the class $\mathbf{MAPA}$. Every $LCA$ group is locally quasi-convex. More generally, every nuclear group is locally quasi-convex, see \cite[Theorem~8.5]{Ban}.
The dual group $\widehat{G}$ of $G$ endowed with the compact-open topology is denoted by $G^{\wedge}$. The homomorphism $\alpha_G : G\to G^{\wedge\wedge} $, $g\mapsto (\chi\mapsto \chi(g))$, is called {\em the canonical homomorphism}. If $\alpha_G$ is a topological isomorphism the group $G$ is called {\em  reflexive}.
In the dual group $\widehat{G}$, we denote by $\sigma(\widehat{G},G)$ the topology of pointwise convergence. Recall that a subset $A$ of $\widehat{G}$ is called {\em equicontinuous} if for every $\e>0$ there is a neighborhood $U$ of zero in $G$ such that
\[
|\chi(x)-1|<\e, \quad \forall x\in U, \; \forall \chi\in A.
\]
We shall use the following fact, see \cite{Nob2}.
\begin{fact} \label{f:Mackey-Ban}
Let $U$ be a neighborhood of zero of an abelian topological group $G$. Then  $U^\triangleright$ is an equicontinuous quasi-convex compact subset of $G^\wedge$. Consequently, a subset $A$ of $G^\wedge$ is equicontinuous if and only if $A\subseteq V^\triangleright$ for some neighborhood $V$ of zero.
\end{fact}

Let $X$ and $Y$ be Tychonoff spaces. We denote by  $C_p(X,Y)$ the space $C(X,Y)$ of all continuous functions from $X$ to $Y$ endowed with the pointwise topology. If $Y=\IR$, set $C(X,\IR):=C(X)$.

A subset $A$ of a topological space $X$ is called
\begin{enumerate}
\item[$\bullet$] {\em relatively compact} if its closure ${\bar A}$ is compact;
\item[$\bullet$] {\em relatively countably compact} if each countably infinite subset in $A$ has a cluster point in $X$;
\item[$\bullet$] {\em relatively sequentially compact} if each sequence in $A$ has a subsequence converging to a point of $X$;
\item[$\bullet$] {\em functionally bounded in $X$} if every $f\in C(X)$ is bounded on $A$.
\end{enumerate}

Recall that a Hausdorff topological space $X$ is called
\begin{enumerate}
\item[$\bullet$] a {\em $k_\omega$-space}  if it is the inductive limit of an increasing sequence $\{ C_n\}_{n\in\NN}$ of its compact subsets;
\item[$\bullet$] a {\em $k_\omega$-group} if $X$ is a topological group whose underlying space is a $k_\w$-space;
\item[$\bullet$] a {\em $\mu$-space} if every functionally bounded subset of $X$ is relatively compact;
\item[$\bullet$] an {\em  $(E)$-space} if its relatively countably compact subsets are relatively compact (\cite{Gro});
\item[$\bullet$] a {\em \v{S}mulyan-space} or a {\em $\check{S}$-space} if its compact subsets are sequentially compact;
\item[$\bullet$] an {\em angelic space} if (1) every relatively countably compact subset of $X$ is relatively compact, and (2) any compact subspace of $X$ is Fr\'{e}chet--Urysohn.
\end{enumerate}

Note that any subspace of an angelic space is angelic, and a subset $A$ of an angelic space $X$ is compact if and only if it is countably compact if and only if $A$ is sequentially compact, see Lemma 0.3 of \cite{Pryce}. Note also that if $\tau$ and $\nu$ are regular topologies on a set $X$ such that $\tau\leq\nu$ and the space $(X,\tau)$ is  angelic, then the space $(X,\nu)$ is also angelic, see \cite{Pryce}.

We need also the following property stronger than the property of being a countably $\mu$-space. A  topological group $G$ is said to have a {\em $\CP$-property} if every separable precompact subset of $X$ has compact closure.
If a topological group $G$ is complete, the closure $\overline{A}$ of each precompact subset $A$ is compact. 
So every complete topological group has the $\CP$-property. However there is a non-complete group with the $\CP$-property which is not a $\mu$-space.

\begin{example} \label{exa:seq-compact-non-compact} {\em
There is a sequentially compact non-compact abelian group $H$ which has the $\CP$-property. Indeed, let $G:= X^\kappa$, where $X$ is a metrizable compact abelian group and the cardinal $\kappa$ is uncountable. For $g=(x_i)_{i\in\kappa} \in G$, set $\mathrm{supp}(g):=\{ i\in\kappa: x_i \not= 0\}$ and define
\[
H:= \{ g\in G: \; |\mathrm{supp}(g)| \leq\aleph_0 \}.
\]
Then $H$ with the induced topology is a proper dense subgroup of $G$. Any countable subset of $H$ is contained in a countable product $Y$ of copies of $X$. Since $Y$  is a compact and metrizable subgroup of $H$, we obtain that the group $H$ is sequentially compact  with the $\CP$-property. It is easy to see that every continuous function on $H$ is bounded.  Therefore $H$ is not compact, and hence $H$ is not a $\mu$-space. Note also that $H$ is Fr\'{e}chet--Urysohn by \cite{Nob}. }
\end{example}

We shall use the following result in which (i) is known but hard to locate explicitly stated.
\begin{lemma} \label{l:func-bounded-precompact}
Let $G$ be a Hausdorff abelian topological group. Then:
\begin{enumerate}
\item[{\rm (i)}] every functionally bounded subset $A$ of $G$ is precompact;
\item[{\rm (ii)}] if $G$ has the $\CP$-property, then a separable subset $B$ of $G$ is functionally bounded if and only if $B$ is precompact.
\end{enumerate}
\end{lemma}

\begin{proof}
(i) If $A$ is not precompact, Theorem 5 of \cite{BGP} implies that $A$ has an infinite uniformly discrete subset $C$, i.e., there is a neighborhood $U$ of zero in $G$ such that $c-c' \not\in U$ for every distinct $c,c'\in C$. So $C$ is not functionally bounded by Lemma 2.1 of \cite{Gab-Top-Nul}, a contradiction.

(ii) follows from (i) and the $\CP$-property.
\end{proof}

Following Orihuela \cite{Orihuela}, a  Hausdorff topological space $X$ is called {\em web-compact}  if there is a nonempty subset $\Sigma$ of $\NN^\NN$ and a family $\{ A_{\alpha}: \alpha\in \Sigma\}$ of subsets of $X$ such that, if
\[
C_{n_1\dots n_k} :=\bigcup \{ A_\beta: \beta=(m_k)\in\Sigma, \; m_1=n_1,\dots,m_k=n_k\}, \quad \forall \alpha=(n_k)\in\Sigma,
\]
the following two conditions hold:
    \begin{enumerate}
    \item[{\rm (i)}] $\overline{\bigcup\{ A_{\alpha}: \alpha\in \Sigma\}}=X$, and
    \item[{\rm (ii)}] if $\alpha=(n_k)\in\Sigma$ and $x_k\in C_{n_1\dots n_k}$ for all $k\in\NN$, then the sequence $\{ x_k\}_{k\in\NN}$ has a cluster point in $X$.
     \end{enumerate}
The class of web-compact spaces is sufficiently rich, see \cite[\S~4.3]{kak}. In particular, every separable space is web-compact.
In what follows we shall use repeatedly the following result, see Proposition 4.2 of \cite{kak}, which follows from a deep result of Orihuela \cite{Orihuela}.

\begin{fact} \label{f:Orihuela-Cp}
If $X$ is web-compact, then the group $C_p(X,\Ss)$  is angelic.
\end{fact}

Following \cite{GKL}, a topological group $G$ is said to have a {\em $\GG$-base} if there is a base $\{ U_\alpha : \alpha\in\NN^\NN\}$ of neighborhoods at the identity such that $U_\beta \subseteq U_\alpha$ whenever $\alpha\leq\beta$ for all $\alpha,\beta\in\NN^\NN$, where $\alpha=(\alpha(n))_{n\in\NN}\leq \beta=(\beta(n))_{n\in\NN}$ if $\alpha(n)\leq\beta(n)$ for all $n\in\NN$. Below we give sufficient conditions on a $MAP$ abelian group $G$ or its dual group $G^\wedge$ to be Bohr angelic.

\begin{proposition}  \label{p:web-compact-Bohr-angelic}
Let $(G,\tau)$ be a  $MAP$ abelian group.
\begin{enumerate}
\item[{\rm (i)}] If $G$ is  web-compact, then $(G^\wedge)^+$ is angelic.
\item[{\rm (ii)}] If $G$ has a $\GG$-base, then $G^+$ is angelic.
\end{enumerate}
\end{proposition}

\begin{proof}
(i) The group $\big(\widehat{G}, \sigma(\widehat{G}, G)\big)$, being a closed subgroup of the group $C_p(G,\Ss)$, is angelic by Fact \ref{f:Orihuela-Cp}. As $\sigma(\widehat{G}, G)\leq \sigma(\widehat{G}, G^{\wedge\wedge})$, we obtain that the group $(G^\wedge)^+ =\big(\widehat{G}, \sigma(\widehat{G}, G^{\wedge\wedge})\big)$ is also angelic.

(ii) Let $\{U_{\alpha}: \alpha\in\mathbb{N}^{\mathbb{N}}\}$ be a $\GG$-base at zero in $G$. Then the family $\{U_{\alpha}^{\triangleright}: \alpha\in\mathbb{N}^{\mathbb{N}}\}$ is a compact resolution in $G^\wedge$ by Theorem 5.1 of \cite{GKL}. Therefore the group $H:= \big( \widehat{G}, \sigma(\widehat{G},G)\big)$ is web-compact by Example 4.1(1) of \cite{kak}. Hence the space $C_p(H,\Ss)$ is angelic by Fact \ref{f:Orihuela-Cp}. So the group $G^+=\big( G, \sigma(G, \widehat{G})\big)$, being a subgroup of $C_p(H,\Ss)$, is also angelic.
\end{proof}

We shall use the following results.
\begin{proposition}[\cite{Gab-Top-Nul}] \label{pSubSchur}
Let $H$ be a subgroup of a MAP Abelian group $X$ and $\mathcal{P}\in \mathfrak{P}_0$.
If $X$ respects $\mathcal{P}$, then $H$ respects $\mathcal{P}$ as well.
\end{proposition}

Let $E$ be a real lcs and $E'$ its topological dual space. Then $E$ is a locally quasi-convex abelian group, see \cite{Ban}.  So it is natural to consider relations between the weak topology $\tau_w :=\sigma(E,E')$ and the Bohr topology $\tau^+:=\sigma(E,\widehat{E})$ on $E$. Denote by $\tau_k$ the compact-open topology on $E'$. The polar of a subset $A$ of $E$ is denoted by $A^\circ :=\{ \chi\in E': |\chi(x)|\leq 1 \, \forall x\in A\}$.
Define
\[
\psi: E' \to \widehat{E},  \; \psi(\chi):= e^{2\pi i \chi},  \quad \big( \mbox{i.e. } \psi(\chi)(x):= e^{2\pi i \chi(x)} \mbox{ for } x\in E\big).
\]

A proof of the next important result can be found in \cite[Proposition~2.3]{Ban}.
\begin{fact} \label{f:dual-E}
Let $E$ be a real lcs and let $\psi: E' \to \widehat{E},  \psi(\chi):= e^{2\pi i \chi}$. Then:
\begin{enumerate}
\item[{\rm (i)}] $\psi$ is an algebraic isomorphism;
\item[{\rm (ii)}] $\psi$ is a topological isomorphism of $(E', \tau_k)$ onto $E^\wedge$.
\end{enumerate}
\end{fact}
We shall say that $\psi$ is the {\em canonical isomorphism } of $E'$ onto $\widehat{E}$. Fact \ref{f:dual-E} implies that  $\tau^+ < \tau_w \leq  \tau$ and hence
\begin{equation} \label{equ:weak-Bohr-property}
\mathcal{P}(E) \subseteq \mathcal{P}(E_w) \subseteq \mathcal{P}(E^+), \quad \mbox{ for every } \mathcal{P}\in \Pp.
\end{equation}

In \cite{ReT} it is proved that $E_w$ and $E^+$ have the same compact sets and hence the same convergent sequences. The next proposition generalizes this result.
\begin{proposition}[\cite{Gab-Top-Nul}] \label{p-Weak=+}
Let $E$ be a real lcs and let $\mathcal{P}\in \mathfrak{P}_0$. Then $\mathcal{P}(E_w)=\mathcal{P}(E^+)$.
\end{proposition}

For an lcs $E$, we denote by $\Bo(E)$ the family of all bounded subsets of $E$. The next assertion complements Proposition \ref{p-Weak=+}. %

\begin{proposition} \label{p:FB-weakly-bound}
If $(E,\tau)$ is a real lcs, then every functionally bounded subset $A$ of $E^+$ is bounded, i.e., $\mathcal{FB}(E^+)\subseteq \Bo(E)$.
\end{proposition}

\begin{proof}
Since $\Bo(E)=\Bo(E_w)$, 
it is sufficient to show that $A$ is weakly bounded. Let $U=[F;\e]$ be a standard weakly open neighborhood of zero in $E$, where $F$ is a finite subset of $E'\setminus\{ 0\}$, $\e>0$ and
\[
[F;\e]:= \{ x\in E: |\chi(x)|<\e \;\; \forall \chi\in F\}.
\]
Fix a $\chi\in F$ and take a $z=z_\chi \in E$ such that $\chi(z)=1$. By Theorem 7.3.5 of \cite{NaB}, we can represent $E$ in the form $E=L_\chi \oplus \ker(\chi)$, where $L_\chi=\spn(z)$ and $\ker(\chi)$ is the kernel of $\chi$. Then $E^+=L_\chi^+ \oplus \ker(\chi)^+$. Since the projection $P_\chi$ from $E$ onto $L_\chi$ is continuous (in $\tau$ and $\tau^+$), $P_\chi(A)$ is a functionally bounded subset of $L_\chi^+ \cong \IR^+$. By \cite{Tri91}, $P_\chi(A)$ is bounded in $L_\chi$. Therefore there exists a $C_\chi>0$ such that $|\chi(a)|<C_\chi$ for every $a\in A$. Set $C:=\max\{ C_\chi: \chi\in F\}$. Then $(\e/C)A\subseteq  U$ since $|\chi\big((\e/C)a\big))|=(\e/C)|\chi(a)|<\e$ for every $a\in A$ and $\chi\in F$. Thus $A$ is bounded.
\end{proof}
We do not know whether there exists a real lcs $E$ such that $\mathcal{FB}(E_w) \subsetneq \mathcal{FB}(E^+)$.

The weak-$\ast$ topology on the dual space of an lcs $E$ plays a crucial role in the theory of locally convex spaces. The next assertion complements Fact \ref{f:dual-E} and Proposition \ref{p-Weak=+} and is used repeatedly in the paper.

\begin{proposition} \label{p:weak*-space-group}
Let $E$ be a real lcs, $\mathcal{P}\in\Pp_0$  and let $\psi:E'\to \widehat{E}$ be the canonical isomorphism. Then:
\begin{enumerate}
\item[{\rm (i)}] the map $\psi : \big(E',\sigma(E',E)\big)^+ \to \big(\widehat{E},\sigma(\widehat{E},E)\big)$
is a topological isomorphism;
\item[{\rm (ii)}] $A\in \mathcal{P}\big(E',\sigma(E',E)\big)$ if and only if $\psi(A)\in\mathcal{P}\big(\widehat{E},\sigma(\widehat{E},E)\big)$;
\item[{\rm (iii)}] the map $\psi : \big(E',\sigma\big(E',(E',\tau_k)'\big)\big)^+ \to \big(\widehat{E},\sigma(\widehat{E},E^{\wedge\wedge})\big)$
is a topological isomorphism;
\item[{\rm (iv)}] $\mathcal{P}\bigg(E',\sigma\big(E',(E',\tau_k)'\big)\bigg)=\mathcal{P}\left(\big(E',\sigma\big(E',(E',\tau_k)'\big)\big)^+\right)$;
\item[{\rm (v)}] $A\in \mathcal{P}\big(E',\sigma\big(E',(E',\tau_k)'\big)\big)$  if and only if $\psi(A)\in\mathcal{P}\big(\widehat{E},\sigma(\widehat{E},E^{\wedge\wedge})\big)$;
\item[{\rm (vi)}] a subset $A$ of $E'$ is equicontinuous if and only if $\psi(A)$ is equicontinuous.
\end{enumerate}
\end{proposition}

\begin{proof}
(i) Since $\big(E',\sigma(E',E)\big)'=E$, Fact \ref{f:dual-E} implies that the dual group of $\big(E',\sigma(E',E)\big)$ can be identified with $E$ under the map $x\mapsto e^{2\pi i x}\, (x\in E)$. So the sets
\[
[F;\e]:=\big\{ \chi\in E': |e^{2\pi i\chi(x)}-1|<\e \; \forall x\in F\big\},
\]
where $F$ is a finite subset of $E$ and $\e>0$, form a base at zero in $\big(E',\sigma(E',E)\big)^+$. The sets
\[
V_{F,\e} :=\big\{ z\in \widehat{E}: |z(x)-1|<\e \;  \forall x\in F\big\},
\]
where $F$ is a finite subset of $E$ and $\e>0$, form a base at zero in $\big(\widehat{E},\sigma(\widehat{E},E)\big)$. Taking into account that $\psi$ is an algebraic isomorphism and $\psi(\chi)(x)=e^{2\pi i\chi(x)}$, we obtain that $\psi\big([F;\e]\big) =V_{F,\e}$. Thus the canonical isomorphism $\psi$ is also a topological isomorphism.

(ii) Set $F:=\big(E',\sigma(E',E)\big)$. Then $F'=E$ and $F_w=F$. Therefore, by Proposition \ref{p-Weak=+}, $\mathcal{P}(F)=\mathcal{P}(F^+)$ and (i) applies.

(iii) Note that $\big(E',\sigma\big(E',(E',\tau_k)'\big)\big)'= (E',\tau_k)'$. Therefore, by Fact  \ref{f:dual-E}, the sets
\[
[F;\e]:=\big\{ \chi\in E': |e^{2\pi i\xi(\chi)}-1|<\e \;  \forall \xi\in F\big\},
\]
where $F$ is a finite subset of $(E',\tau_k)'$ and $\e>0$, form a base at zero in $\big(E',\sigma\big(E',(E',\tau_k)'\big)\big)^+ $. Analogously, the sets
\[
W_{F,\e} :=\big\{ z\in \widehat{E}: |s(z)-1|<\e \;  \forall s\in F\big\},
\]
where $F$ is a finite subset of $E^{\wedge\wedge}$ and $\e>0$, form a base at zero in $\big(\widehat{E},\sigma(\widehat{E},E^{\wedge\wedge})\big)$. By (i) of Fact  \ref{f:dual-E}, the map
\[
\widehat{\psi}: (E',\tau_k)' \to (E',\tau_k)^{\wedge}, \quad \widehat{\psi}(\xi):= e^{2\pi i\xi},
\]
is an algebraic isomorphism and, by (ii) of Fact  \ref{f:dual-E}, the adjoint map $\psi^\ast$ of $\psi$
\begin{equation} \label{equ:exponent-21}
\psi^\ast : E^{\wedge\wedge} \to (E',\tau_k)^\wedge, \quad (\psi^\ast(\eta),\chi)=\big(\eta,\psi(\chi)\big), \; \eta\in E^{\wedge\wedge}, \chi\in E',
\end{equation}
is a topological isomorphism. In particular, for $\alpha=\psi^\ast(\eta)$, (\ref{equ:exponent-21}) implies
\begin{equation} \label{equ:exponent-22}
(\alpha,\chi)=\big( (\psi^\ast)^{-1} (\alpha), \psi(\chi)\big), \quad \forall \alpha\in (E',\tau_k)^{\wedge},\; \forall \chi\in E'.
\end{equation}
So the map  $H:= (\psi^\ast)^{-1}\circ \widehat{\psi}: (E',\tau_k)' \to E^{\wedge\wedge}$ is an algebraic isomorphism such that, for every $z=\psi(\chi)\in E^\wedge$ with $\chi\in E'$ and each $\xi\in (E',\tau_k)'$, we have
\[
\big( H(\xi),z\big)= \left( (\psi^\ast)^{-1}\circ \widehat{\psi} (\xi), \psi(\chi) \right) \stackrel{(\ref{equ:exponent-22})}{=} \left( \widehat{\psi}(\xi),\chi \right)= e^{2 \pi i\xi(\chi)} =e^{2\pi i \xi\big( \psi^{-1}(z)\big)}.
\]
Therefore, for a finite subset $F$ of $(E',\tau_k)'$ and $\e >0$, we obtain
\[
\begin{split}
\psi\big( [F;\e]\big) & = \left\{ z\in \widehat{E}: \left| e^{2\pi i \xi\big( \psi^{-1}(z)\big)} -1 \right| <\e \; \forall \xi\in F \right\} \\
& = \left\{ z\in \widehat{E}: \left| \big( H(\xi), z\big) -1 \right| <\e \; \forall \xi\in F \right\} = W_{H(F),\e}.
\end{split}
\]
Thus $\psi$ is a topological isomorphism.

(iv) follows from Proposition \ref{p-Weak=+} applied to the space $G=G_w := \big(E',\sigma\big(E',(E',\tau_k)'\big)\big) $, and
(v) follows from (iii) and (iv).

(vi) We shall use the following easily checked inequalities
\begin{equation} \label{equ:exponent-1}
\pi |\phi| \leq \big|e^{2\pi i \phi} -1\big| \leq 2\pi|\phi|, \quad \phi\in [-1/2,1/2].
\end{equation}

Let $A\subseteq E'$ be equicontinuous. For every $0<\e<0.1$, take a neighborhood $U$ of zero in $E$ such that
\begin{equation} \label{equ:exponent-2}
|a(x)|<\e, \quad \forall a\in A, \; \forall x\in U.
\end{equation}
Then (\ref{equ:exponent-1}) and (\ref{equ:exponent-2}) imply
\[
|\psi(a)(x) -1| =\big|e^{2\pi i a(x)} -1\big| \leq 2\pi\e, \quad \forall a\in A, \; \forall x\in U.
\]
Thus $\psi(A)$ is equicontinuous.

Conversely, let $\psi(A)$ be equicontinuous. For every $0<\e<0.1$, take an absolutely convex neighborhood $U$ of zero in $E$ such that
\[
\big|e^{2\pi i a(x)} -1\big| < \e, \quad \forall a\in A, \; \forall x\in U.
\]
If $a(x)= t+m$ with $t\in[-1/2,1/2]$ and $m\in\ZZ$, (\ref{equ:exponent-1}) implies $\pi|t| \leq |e^{2\pi it} -1|=|e^{2\pi ia(x)} -1|<\e$, and hence
\begin{equation} \label{equ:exponent-3}
a(x)\in (-\e/\pi,\e/\pi) +\ZZ,  \quad \forall a\in A, \; \forall x\in U.
\end{equation}
Since $U$ is arc-connected, $0\in U$ and $\e<0.1$, (\ref{equ:exponent-3}) implies
\[
a(x)\in (-\e/\pi,\e/\pi) ,  \quad \forall a\in A, \; \forall x\in U.
\]
Thus $A$ is equicontinuous.
\end{proof}


\section{ General results} \label{sec:2}


To show that the Glicksberg property and the Schur property can be naturally defined by two  functors in the category $\mathbf{TG}$ we consider two classes of topological groups introduced by Noble in \cite{Nob, Nob2}, namely, the classes of $k$-groups and $s$-groups. 

For every $(X,\tau)\in \mathbf{TG}$ denote by $k_g(\tau)$ the finest group topology for $X$ coinciding on compact sets with $\tau$. In particular, $\tau$ and $ k_g(\tau)$ have the same family of compact subsets. Clearly, $\tau\leq k_g(\tau)$. If $\tau = k_g(\tau)$, the group $(X,\tau)$ is called a {\it $k$-group} \cite{Nob2}. The group $(X,k_g(\tau))$ is called the {\it $k_g$-modification} of $X$. The assignment $\mathbf{k}_g(X,\tau) := (X,k_g(\tau))$ is a functor from $\mathbf{TG}$ to the full subcategory $\mathbf{K}$ of all  $k$-groups. The class $\mathbf{K}$   contains all topological groups whose underlaying space is a $k$-space. In particular, the class $\mathbf{LC}$ ($\mathbf{LCA}$) of all locally compact (and abelian, respectively) groups  is contained in $\mathbf{K}$. Since every metrizable group is a $k$-space we have $\mathbf{LC} \subsetneqq  \mathbf{K}$. The family of all abelian $k$-groups we denote by $\mathbf{KA}$.
Denote by $\mathfrak{RC}$ the class of all $MAP$ groups which respect compactness.

Similar to $k$-groups  we define $s$-groups (we follow \cite{Ga30}).
Let $(X,\tau)$ be a (Hausdorff) topological group and let $S$ be the set of all sequences in $(X,\tau)$ converging to the unit $e\in X$. Then there exists the finest Hausdorff group topology $\tau_S$ on the underlying group $X$ in which all sequences of $S$ converge to $e$.  If $\tau =\tau_S$, the group $X$ is called an {\it $s$-group}. The assignment $\mathbf{s}_g(X,\tau) := (X,\tau_S)$ is a functor from $\mathbf{TG}$ to the full subcategory $\mathbf{S}$ of all  $s$-groups. The class $\mathbf{S}$   contains all sequential groups \cite[1.14]{Ga30}.  Note that $X$ and $\mathbf{s}_g(X)$ have the same set of convergent sequences \cite[4.2]{Ga30}. The family of all abelian $s$-groups we denote by $\mathbf{SA}$.
Every $s$-group is also a $k$-group \cite{Ga3}, so $\mathbf{S}\subseteq \mathbf{K}$ and $\mathbf{SA}\subseteq \mathbf{KA}$.
Denote by  $\mathfrak{RS}$ the class of all $MAP$  groups which have the Schur property.

For a topological group $X$ with the identity $e$, set
\[
c_0 (X) := \left\{ (x_n)_{n\in\mathbb{N}} \in X^\mathbb{N} : \; \lim_{n} x_n = e \right\}
\]
and let $\mathfrak{u}_0$ be the uniform topology on $c_0(X)$ generated by the sets of the form $V^\NN$, where $V$ is an open neighborhood of $e\in X$. Set $\mathfrak{F}_0 (X) :=(c_0 (X), \mathfrak{u}_0)$.

In (1) and (3) of the next proposition we give  categorical characterizations of the Schur property and the Glikcsberg property, note also that  (8) generalizes Theorem 1.2 of \cite{Tri91}. 

\begin{proposition} \label{p:Schur-categorical}
Let $X$ and $Y$ be $MAP$ topological groups.
\begin{enumerate}
\item[{\rm (1)}] $X\in \mathfrak{RC}$ if and only if $(\mathbf{k}_g\circ\mathfrak{B})(X) =\mathbf{k}_g(X)$.
\item[{\rm (2)}] $X\in  \mathbf{K}\cap \mathfrak{RC}$ if and only if $(\mathbf{k}_g\circ\mathfrak{B})(X) =X$.
\item[{\rm (3)}] $X\in \mathfrak{RS}$ if and only if $(\mathbf{s}_g\circ\mathfrak{B})(X) =\mathbf{s}_g(X)$.
\item[{\rm (4)}] $X\in  \mathbf{S}\cap \mathfrak{RS}$ if and only if $(\mathbf{s}_g\circ\mathfrak{B})(X) =X$.
\item[{\rm (5)}] $\mathbf{PCom}\subsetneqq  \mathfrak{RC}$ and $\mathbf{LCA} \subsetneqq  \mathbf{KA}\cap \mathfrak{RC}$.
\item[{\rm (6)}] $\mathbf{K}\cap \mathfrak{RC} \subsetneqq \mathbf{K}$ and $\mathbf{K}\cap \mathfrak{RC}\subsetneqq \mathfrak{RC}$.
\item[{\rm (7)}] $\mathbf{S}\cap \mathfrak{RS} \subsetneqq \mathbf{S}$ and $\mathbf{S}\cap \mathfrak{RS}\subsetneqq \mathfrak{RS}$.
\item[{\rm (8)}] Let $X\in \mathbf{K}$ and $Y\in \mathfrak{RC}$ and let $\phi: X\to Y$ be a homomorphism. If $\phi^+: X^+ \to Y^+, \phi^+(x):=\phi(x),$ is continuous, then $\phi$ is continuous.
\item[{\rm (9)}] Let $X\in \mathbf{S}$ and $Y\in \mathfrak{RS}$ and let $\phi: X\to Y$ be a homomorphism. If $\phi^+: X^+ \to Y^+, \phi^+(x):=\phi(x),$ is continuous, then $\phi$ is continuous.
\end{enumerate}
\end{proposition}

\begin{proof}
(1) If $X\in \mathfrak{RC}$, then $(\mathbf{k}_g\circ\mathfrak{B})(X) =\mathbf{k}_g(X)$ by the definition of the respecting compactness  and the definition of $\mathbf{k}_g(X)$.
Conversely, let $(\mathbf{k}_g\circ\mathfrak{B})(X) =\mathbf{k}_g(X)$ and let $K$ be compact in $\mathfrak{B}(X)$. Then $K$ is compact in $(\mathbf{k}_g\circ\mathfrak{B})(X)$ by the definition of $k_g$-modification. So $K$ is compact in $\mathbf{k}_g(X)$. Hence,  by the definition of $k_g$-modification,  $K$ is compact  in $X$. Thus $X\in \mathfrak{RC}$.

(2) If $X\in  \mathbf{K}\cap \mathfrak{RC}$, then (1) and the definition of $k$-groups imply $(\mathbf{k}_g\circ\mathfrak{B})(X) =\mathbf{k}_g(X)=X$.
Conversely, let $(\mathbf{k}_g\circ\mathfrak{B})(X) =X$. Since $\mathbf{k}_g\circ\mathbf{k}_g=\mathbf{k}_g$, the equalities
\[
\mathbf{k}_g(X)=\mathbf{k}_g\circ(\mathbf{k}_g\circ \mathfrak{B}(X)) =(\mathbf{k}_g\circ \mathfrak{B})(X) =X
\]
and  (1) imply that $X$ is a $k$-group and $X\in \mathfrak{RC}$.

(3) and (4) can be proved analogously to (1) and (2), respectively.



(5) Since $\mathfrak{B}(K)=K$ for each precompact group $K$, the first inclusion follows. The second one holds by the Glicksberg theorem. To prove that these  inclusions are strict take an arbitrary compact totally disconnected metrizable group $X$. Then $\mathfrak{F}_0(X)$ is metrizable, and hence it is a $k$-group.  By Theorem 1.3 of \cite{Gab-Top-Nul}, $\mathfrak{F}_0(X)$ respects compactness and it is not locally precompact by \cite{Ga8}. Thus the inclusions  are strict.

(6)-(7) Being metrizable the group $\mathfrak{F}_0(\TT)$ belongs to $\mathbf{SA}\subseteq \mathbf{KA}$ (here $\TT=\IR/\ZZ$). However, $\mathfrak{F}_0(\TT)$ does not respect compactness and convergent sequences  by Theorem 1.3 of \cite{Gab-Top-Nul}. Thus $\mathbf{K}\cap \mathfrak{RC} \not= \mathbf{K}$ and $\mathbf{S}\cap \mathfrak{RS} \not= \mathbf{S}$.

To prove that the second inclusion is proper it is enough to find a precompact abelian group $X$ which is not a $k$-group, and hence it is not an $s$-group.
Take an arbitrary non-measurable subgroup $H$ of $\TT$ and set $X:=(\ZZ, T_H)$, where $T_H$ is the smallest group topology on $\ZZ$ for which the elements of $H$ are continuous. Then  the precompact group $X$ does not contain non-trivial convergent sequences  (see \cite{CRT}). Since $X$ is countable, we obtain that $X$ also does not have infinite  compact subsets by  \cite[3.1.21]{Eng}. This immediately implies that the $k_g$-modification of $X$ is discrete. Hence $\mathbf{k}_g(X)=\ZZ_d$ is an infinite discrete $LCA$ group. So $\mathbf{k}_g(X)\not= X$ and $X$ is not a $k$-group. Thus the second inclusion is  proper.

(8) Let $id_X: X\to X^+$ and $id_Y: Y\to Y^+$ be the identity continuous maps. Fix arbitrarily   a compact subset $K$ in $X$. Then $K^+ :=\phi^+(id_X(K))$ is compact in $Y^+$. As $Y\in \mathfrak{RC}$, $K^+$ is compact in $Y$. So $id_Y|_{K^+}$ is a homeomorphism.   Hence $\phi|_K = (id_Y|_{K^+})^{-1} \circ \phi^+\circ (id_X |_K)$ is continuous. So $\phi$ is continuous on any compact subset of $X$. As $X$ is a $k$-group, $\phi$ is continuous (see \cite{Nob2}).

(9) is proved analogously to (8). 
\end{proof}

\begin{remark} {\em
The fact that $(G,\mathcal{T}^+)$ is precompact whenever $(G,\mathcal{T})$ is a $MAP$ abelian group suggests the following two natural questions posed in \cite[1.2]{CTW}  (see also \cite{Tri91}):

(i) Let $(G,\mathcal{U})$ be an abelian precompact group. Must there exist a topological group topology $\mathcal{T}$ for $G$ such that $(G,\mathcal{T})$ is a $LCA$ group and $\mathcal{U}=\mathcal{T}^+$?

(ii) Let $G$ be an abelian group with  topological group topologies $\mathcal{T}$ and $\mathcal{U}$ such that $(G,\mathcal{T})$ is a $LCA$ group, $(G,\mathcal{U})$ is an abelian precompact group, $\mathcal{U}\subseteq \mathcal{T}$, and a subset $A\subseteq G$ is $\mathcal{T}$-compact if and only if $A$ is $\mathcal{U}$-compact. Does it follow that $\mathcal{U}=\mathcal{T}^+$?

In \cite{CTW}, the authors showed that the answer to both these questions is ``no''. Let us show that the group $X$ in the proof of (6)-(7) of Proposition \ref{p:Schur-categorical} also  answers negatively  to these questions. Set $G=\mathbb{Z}$ and  $\mathcal{U}=T_H$. Since $G$ is countable, every locally compact group topology $\mathcal{T}$ on $G$ must be discrete. So $\mathcal{T}^+ = T_\mathbb{T}$ and $\mathcal{U}\subseteq \mathcal{T}$. Further, as it was noticed in the proof of (6)-(7), a subset $A$ of $G$ is $\mathcal{T}$-compact if and only if $A$ is $\mathcal{U}$-compact (if and only if $A$ is finite). However, since $H \not= \TT$, we obtain $\mathcal{U} \not= \mathcal{T}^+$ by \cite{CoR}. }
\end{remark}

We note the following assertion.
\begin{proposition} \label{p:mu-space-Schur}
Let $(G,\tau)$ be a $MAP$  group such that every functionally bounded subset of $G^+$ has compact closure in $G$. Then $G$ respects all properties $\mathcal{P}\in\Pp$ and $G^+$ is a $\mu$-space.
\end{proposition}

\begin{proof}
Let $A\in\mathcal{P}(G^+)$. Then $A$ is functionally bounded in $G^+$. Therefore its $\tau$-closure $\overline{A}$ is compact in $G$, so  the identity map $id: \big(\overline{A}, \tau|_{\overline{A}}\big) \to \big(\overline{A}, \tau^+|_{\overline{A}}\big)$ is a homeomorphism. Hence $G^+$ is a $\mu$-space and $A\in \mathcal{P}(G)$.   Thus $G$ respects $\mathcal{P}$.
\end{proof}

It is clear that the Glicksberg property implies the Schur property, but as we mentioned in the introduction, the converse is not true in general. Some other relations between respecting properties  are given in the next proposition, which gives a partial answer to Problem 7.2 in \cite{Gab-Top-Nul}.
\begin{proposition} \label{p:Schur=sequential-comp}
Let $(G,\tau)$ be a $MAP$ group. Then:
\begin{enumerate}
\item[{\rm (i)}] $G$  has the Schur property if and only if it respects sequential compactness;
\item[{\rm (ii)}] if $G$ respects countable compactness, then $G$ has the Schur property;
\item[{\rm (iii)}] if $G$ respects pseudocompactness, then $G$ has the Schur property;
\item[{\rm (iv)}] if $G$ is a countably $\mu$-space and respects functional boundedness, then $G$ has the Schur property;
\item[{\rm (v)}] if $G$ is complete and respects countable compactness, then $G$ has the Glicksberg property;
\item[{\rm (vi)}] if $G$ is complete and respects pseudocompactness, then $G$ respects countable compactness;
\item[{\rm (vii)}] if $G$ is complete and respects functional boundedness, then $G$ respects all properties $\mathcal{P}\in\Pp$.
\end{enumerate}
\end{proposition}

\begin{proof}
(i) (If $G$ is an abelian group the necessity is proved in Proposition 23 of \cite{BMPT}.) Assume that $(G,\tau)$ has the Schur property and let $A$ be a sequentially compact subset of $G^+$. Take a sequence $S=\{ a_n\}_{n\in\NN}$ in $A$. Then $S$ has a $\tau^+$-convergent subsequence $S'$. By the Schur property $S'$ converges in $\tau$. Hence $A$ is $\tau$-sequentially compact. Thus $G$ respects sequential compactness.

Conversely, assume that $(G,\tau)$ respects sequential compactness  and let $\{ a_n\}_{n\in\NN}$ be a sequence $\tau^+$-converging to an element $a_0\in G$. Set $S:= \{ a_n\}_{n\in\NN} \cup\{ a_0\}$, so $S$ is $\tau^+$-compact. Being countable $S$ is metrizable and hence $\tau^+$-sequentially compact. So $S$ is sequentially compact in $\tau$. We show that $a_n\to a_0$ in $\tau$. Suppose for a contradiction that there is a $\tau$-neighborhood $U$ of $a_0$ which does not contain an infinite subsequence $S'$ of $S$. Then there is a subsequence   $\{ a_{n_k}\}_{k\in\NN}$ of $S'$ which $\tau$-converges to an element $g\in S$. Clearly, $g\not= a_0$ and $a_{n_k}\to g$ in the Bohr topology, and hence $a_n\not\to a_0$ in $\tau^+$, a contradiction. Therefore $a_n\to a_0$ in $\tau$. Thus $G$  has the Schur property.

(ii),(iii) Let $\{ a_n\}_{n\in\NN}$ be a sequence $\tau^+$-converging to an element $a_0\in G$. Set $S:= \{ a_n\}_{n\in\NN} \cup\{ a_0\}$, so $S$ is $\tau^+$-compact. Hence $S$ is $\tau^+$-countably compact. So $S$ is countably compact or pseudocompact in $\tau$, respectively. As any countable space is normal, in both cases $S$ is countably compact in $\tau$. We show that $a_n\to a_0$ in $\tau$. Suppose for a contradiction that there is a $\tau$-neighborhood $U$ of $a_0$ which does not contain an infinite subsequence $S'$ of $S$. Then $S'$ has a $\tau$-cluster point $g\in S$ and clearly $g\not= a_0$. Note that $g$ is also a cluster point of $S'$ in the Bohr topology $\tau^+$. Hence $g=a_0$, a contradiction. Therefore $a_n\to a_0$ in $\tau$. Thus $G$  has the Schur property.

(iv) Let $S=\{ a_n: n\in\NN\} \cup \{ a_0\}$ be a sequence in $G^+$ which $\tau^+$-converges to $a_0$. Since $S$ is also functional bounded in $G^+$, we obtain that $S$ is closed and functionally bounded in $G$.  So $S$ is compact in $G$ because $G$ is a countably $\mu$-space. As the identity map $(S,\tau|_S)\to (S,\tau^+|_S)$ is a homeomorphism, $a_n\to a_0$ in $G$. Thus $G$ has the Schur property.

(v) Let $K$ be a compact subset of $G^+$. Then $K$ is countably compact in $G^+$ and hence in $G$. Since functionally bounded subsets are precompact by Lemma \ref{l:func-bounded-precompact}, the completeness of $G$ and the closeness of $K$ in $G$ imply that $K$ is a compact subset of $G$. Thus $G$ has the Glicksberg property.

(vi) Let $A$ be a countably compact subset of $G^+$. Then $A$ is pseudocompact in $G^+$ and hence in $G$. The completeness of $G$ and Lemma \ref{l:func-bounded-precompact} imply that the closure $\overline{A}$ of $A$ in $G$ is compact. As the identity map $(\overline{A},\tau|_{\overline{A}})\to (\overline{A},\tau^+|_{\overline{A}})$ is a homeomorphism, we obtain that $A$ is countably compact in $G$. Thus $G$ respects countable compactness.

(vii) This is Theorem 1.2 of \cite{Gab-Top-Nul}.
\end{proof}
Proposition \ref{p:Schur=sequential-comp} shows that the Schur property is the weakest one among the properties of $\Pp_0$.
We do not know whether the completeness of $G$ in (v)-(vii) of Proposition \ref{p:Schur=sequential-comp} can  be dropped. Also we do not know an example of a $MAP$ group which respects countable compactness but does not have the Glicksberg property (or respects pseudocompactness but does not respect countable compactness, etc.).



Recall that in a complete group the class of precompact sets and the class of functionally bounded sets are coincide. Recall also that if $G$ is a complete $g$-group, then $G$ and $G^+$ are $\mu$-spaces, see Theorem 3.2 of \cite{HM}. Therefore the next theorem generalizes Theorem 3.3 of \cite{HM}, cf. also Theorem 1.2 of \cite{Gab-Top-Nul}.
\begin{theorem} \label{t:Glicksberg-respecting}
Let $(G,\tau)$ be a $MAP$ group such that $G^+$ is a $\mu$-space. Then the following assertions are equivalent:
\begin{enumerate}
\item[{\rm (i)}] $G$  respects compactness;
\item[{\rm (ii)}] $G$  respects countable compactness and $G$ is a $\mu$-space;
\item[{\rm (iii)}] $G$  respects pseudocompactness and  $G$ is a $\mu$-space;
\item[{\rm (iv)}] $G$  respects functional boundedness   and  $G$ is a $\mu$-space;
\item[{\rm (v)}] $G$ is a $\mu$-space and every non-functionally bounded subset $A$ of $G$ has an infinite subset $B$ which is discrete and $C$-embedded in $G^+$.
\end{enumerate}
If (i)-(v) hold, then  every functionally bounded subset in $G^+$ is relatively compact in $G$.
\end{theorem}

\begin{proof}
(i)$\Rightarrow$(ii) Let $A$ be a  countably compact subset of $G^+$.  As $G^+$ is a $\mu$-space, the $\tau^+$-closure $\overline{A}$ of $A$ is compact in $G^+$. Therefore $\overline{A}$ is compact in $G$ by the Glicksberg property, and hence $A$ is relatively compact in $G$. Since the identity map $(\overline{A},\tau|_{\overline{A}})\to (\overline{A},\tau^+|_{\overline{A}})$ is a homeomorphism, we obtain that $A$ is  countably compact in $G$. Thus $G$ respects countable compactness. The same proof shows that  every functionally bounded subset in $G^+$ is relatively compact in $G$, and in particular $G$ is a $\mu$-space.

(ii)$\Rightarrow$(iii) Let $A$ be a pseudocompact subset of $G^+$. Then the closure $K$ of $A$ in $G^+$ is $\tau^+$-compact because $G^+$ is a $\mu$-space. Therefore $K$ is countably compact in $G$. Being closed $K$ also is compact in $G$ since $G$ is a $\mu$-space. Since the identity map $(K,\tau|_{K})\to (K,\tau^+|_{K})$ is a homeomorphism, we obtain that $A$ is pseudocompact in $G$. Thus $G$ respects pseudocompactness.

The implication (iii)$\Rightarrow$(iv) is proved analogously to (ii)$\Rightarrow$(iii).

(iv)$\Rightarrow$(v) Let $A$ be a non-functionally bounded  subset of $G$. As $G$  respects functional boundedness it follows that  $A$ is not functionally bounded in $G^+$. Let $f$ be a continuous function on $G^+$ which is unbounded on $A$. If we take $B$ as a sequence $\{ a_{n}\}_{n\in\NN}$ in $A$ such that $|f(a_{n+1})|> |f(a_{n})| +1$ for all $n\in\NN$, then $B$ is discrete and $C$-embedded in $G^+$.

(v)$\Rightarrow$(i) Let $K$ be a compact subset of $G^+$. Then $K$ must be functionally bounded in $G$. Since $G$ is a $\mu$-space and $K$ is also closed in $G$ we obtain that $K$ is compact in $G$. Thus $G$ respects compactness.
\end{proof}

In several important classes of $MAP$ groups some of the properties from $\Pp_0$ hold simultaneously.
\begin{proposition}  \label{p:Schur-E-S}
Let $(G,\tau)$ be a complete $MAP$ group.
\begin{enumerate}
\item[{\rm (i)}] If $G^+$ is an $(E)$-space, then $G$ has the Glicksberg property  if and only if  $G$ respects countable compactness.
\item[{\rm (ii)}] If $G^+$ is a $\check{S}$-space, then $G$ has the Schur property if and only if $G$ has the Glicksberg property.
\end{enumerate}
\end{proposition}

\begin{proof}
(i) Let $G$ have the Glicksberg property and let $K$ be a countably compact subset of $G^+$. Since $G^+$ is an $(E)$-space, $K$ is relatively compact in $G^+$, and hence its $\tau^+$-closure $\overline{K}$ is compact in $G^+$. So $\overline{K}$ is compact in $G$  and the identity map $(\overline{K},\tau|_{\overline{K}}) \to (\overline{K},\tau^+|_{\overline{K}})$ is a homeomorphism. Therefore  $K$ is countably compact in $G$. Thus $G$ respects countable compactness. The converse assertion follows from (v) of Proposition \ref{p:Schur=sequential-comp}.

(ii) Let $G$ have the Schur property and let $K$ be a compact subset of $G^+$. Then $K$ is sequentially compact in $G^+$, and hence in $G$ by (i) of Proposition \ref{p:Schur=sequential-comp}. As $K$ is precompact and closed in $G$ and $G$ is complete, we obtain that $K$ is compact in $G$. Thus $G$ has the Glicksberg property. The converse assertion is clear.
\end{proof}

For Bohr angelic groups we obtain the following result. 
\begin{theorem}  \label{t:Schur-Bohr-angelic}
Let $(G,\tau)$ be a  $MAP$ group. If $G^+$ is angelic, then the following assertions are equivalent:
\begin{enumerate}
\item[{\rm (i)}] $G$ has the Schur property;
\item[{\rm (ii)}] $G$ has the Glicksberg property;
\item[{\rm (iii)}] $G$ respects  sequential compactness;
\item[{\rm (iv)}] $G$ respects  countable compactness.
\end{enumerate}
If, in addition, $G$ is a countably $\mu$-space, then (i)-(iv) are equivalent to
\begin{enumerate}
\item[{\rm (v)}] every non-functionally bounded subset of $G$ has an infinite subset which is closed and discrete in $G^+$.
\end{enumerate}
\end{theorem}

\begin{proof}
The equivalence (i)$\Leftrightarrow$(iii) and the implications (ii)$\Rightarrow$(i) and (iv)$\Rightarrow$(i) follow from Proposition \ref{p:Schur=sequential-comp}. 

(iii)$\Rightarrow$(ii),(iv) Let $K$ be a compact subset or a countably compact subset of $G^+$. Then $K$ is sequentially compact in $G^+$ by \cite[Lemma~0.3]{Pryce}, and hence $K$ is sequentially compact in $G$. Since $G$ is also angelic, we obtain that $K$ is compact or countably compact in $G$. Thus $G$ has the Glicksberg property or respects countable compactness, respectively.

(ii)$\Rightarrow$(v) Suppose for a contradiction that there is a non-functionally bounded subset $A$ in $G$ whose every countably infinite subset is either non-closed in $G^+$ or is not discrete in $G^+$. So, in both cases, every countably infinite subset of $A$ has a cluster point in $G^+$. Therefore $A$ is relatively countably compact in $G^+$. The angelicity of $G^+$ implies that the closure $\overline{A}$ of $A$ in $G^+$ is compact in $G^+$. Hence $\overline{A}$ is compact in $G$ by the Glicksberg property. Thus $A$ is functionally bounded in $G$, a contradiction.

(v)$\Rightarrow$(i) Let $g_n\to e$ in $G^+$, where $e$ is the identity of $G$. Set $S:=\{ g_n\}_{n\in\NN} \cup \{ e\}$, so $S$ is a compact subset of $G^+$. Let us show that $S$ is functionally bounded in $G$. Indeed, otherwise, there would exist a subsequence $\{ g_{n_k}\}_{k\in\NN}$ of $S$ which is closed and discrete in $G^+$. Then $g_n$ does not converge to $e$ in $G^+$, a contradiction. So $S$ is functionally bounded in $G$. Thus the set $S$ being also countable and closed in $G$  is compact in $G$ (recall that $G$ is a countably $\mu$-space). Therefore the identity map $(S,\tau|_S)\to (S,\tau^+|_S)$ is a homeomorphism. Hence $g_n\to e$ in $G$. Thus $G$ has the Schur property.
\end{proof}

Now we are ready to prove   Theorem \ref{t:G-base-Schur}.
\begin{proof}[Proof of Theorem \ref{t:G-base-Schur}]
The equivalences (i)$\Leftrightarrow$(ii)$\Leftrightarrow$(iii)$\Leftrightarrow$(iv)$\Leftrightarrow$(v) follow from Theorem \ref{t:Schur-Bohr-angelic} and (ii) of Proposition \ref{p:web-compact-Bohr-angelic}. The implications (vi)$\Rightarrow$(i) and  (vii)$\Rightarrow$(i) follow from (iii) and (iv) of Proposition \ref{p:Schur=sequential-comp}, respectively. Finally, the implications (ii)$\Rightarrow$(vi) and  (ii)$\Rightarrow$(vii) follow from Theorem \ref{t:Glicksberg-respecting}.
\end{proof}

\begin{corollary}  \label{c:G-base-Schur-Lind}
For a Lindel\"{o}f $MAP$ abelian group $G$ with a $\GG$-base the following assertions are equivalent:
\begin{enumerate}
\item[{\rm (i)}] there is $\mathcal{P}\in\Pp$ such that $G$ respects $\mathcal{P}$;
\item[{\rm (ii)}] $G$ respects all properties $\mathcal{P}\in\Pp$;
\item[{\rm (iii)}] every non-functionally bounded subset of $G$ has an infinite subset which is closed and discrete in $G^+$.
\end{enumerate}
If, in addition, $G$ has the $\CP$-property, then (i)-(iii) are equivalent to
\begin{enumerate}
\item[{\rm (iv)}] every non-precompact sequence in $G$ has an infinite subsequence which is closed and discrete in $G^+$.
\end{enumerate}
\end{corollary}

\begin{proof}
Since $G$ is Lindel\"{o}f, the group $G^+$ is also Lindel\"{o}f. Therefore $G$ and $G^+$ are $\mu$-spaces. Now Theorem \ref{t:G-base-Schur} implies the equivalences (i)$\Leftrightarrow$(ii)$\Leftrightarrow$(iii). If $G$ has the $\CP$-property, the equivalence (iii)$\Leftrightarrow$(iv) follows from Lemma \ref{l:func-bounded-precompact}.
\end{proof}



\section{Real locally convex spaces and respecting properties} \label{sec:Free-Schur}


Following \cite{Mar} (\cite{GM}),  the {\em  free locally convex space}  $L(X)$  (the {\em  free topological vector space} $\VV(X)$) on a Tychonoff space $X$ is a pair consisting of a locally convex space $L(X)$ (a topological vector space $\VV(X)$, resp.) and  a continuous map $i: X\to L(X)$ ($i: X\to \VV(X)$, resp.)  such that every  continuous map $f$ from $X$ to a locally convex space  $E$ (a topological vector space $E$, resp.) gives rise to a unique continuous linear operator ${\bar f}: L(X) \to E$ (${\bar f}: \VV(X) \to E$)  with $f={\bar f} \circ i$. The free locally convex space $L(X)$ and the free topological vector space $\VV(X)$ always exist and are essentially unique. The set $X$ forms a Hamel basis for $L(X)$ and $\VV(X)$, and  the map $i$ is a topological embedding, see \cite{GM,Rai,Usp2}.

For a Tychonoff space  $X$, let $\CC(X)$ be the space $C(X)$ endowed with the compact-open topology $\tau_k$. Then the sets of the form
\[
[K;\e]:=\{ f\in C(X): |f(x)|<\e \; \forall x\in K\}, \mbox{ where } K \mbox{ is compact and } \e>0,
\]
form a base of open neighborhoods at zero in $\tau_k$.

Denote by $M_c(X)$ the space of all real regular Borel measures on $X$ with compact support. It is well-known that the dual space of $\CC(X)$ is $M_c(X)$, see \cite{Jar}. Denote by $\tau_e$ the polar topology on $M_c(X)$ defined by the family of all equicontinuous  pointwise bounded subsets of $C(X)$.
We shall use the following deep result of Uspenski\u{\i} \cite{Usp2}.
\begin{theorem}[\cite{Usp2}] \label{t:Free-complete-L}
Let $X$ be a Tychonoff space and let $\mu X$ be the Dieudonn\'{e} completion of $X$. Then the completion $\overline{L(X)}$ of $L(X)$ is topologically isomorphic to $\big(M_c(\mu X),\tau_e\big)$. Consequently, $L(X)$ is complete if and only if $X$ is Dieudonn\'{e} complete and does not have infinite compact subsets.
\end{theorem}

\begin{corollary} \label{p:Ck-Mc-compatible}
Let $X$ be a Dieudonn\'{e} complete space. Then the topology $\tau_e$ on $M_c(X)$ is compatible with the duality $(\CC(X),M_c(X))$.
\end{corollary}
\begin{proof}
It is well-known that $L(X)'=C(X)$, see \cite{Rai}. Now Theorem \ref{t:Free-complete-L} implies $(M_c(X),\tau_e)'=L(X)'=C(X)$.
\end{proof}

We need  also the following fact, see \S 5.10 in \cite{NaB}. 
\begin{proposition} \label{p:Ascoli-Free-Ck}
Let $X$ be a Tychonoff space and let $\KK$ be an equicontinuous pointwise bounded subset of $C(X)$. Then the pointwise closure ${\bar A}$ of $A$ is $\tau_k$-compact and equicontinuous.  
\end{proposition}

Following \cite{BG},  a Tychonoff space $X$ is called {\em Ascoli} if every compact subset $\KK$ of $\CC(X)$ is equicontinuous. Note that $X$ is Ascoli if and only if the canonical map $L(X)\to \CC(\CC(X))$ is an embedding of locally convex spaces, see \cite{Gabr-LCS-Ascoli}. Below we give another characterization of Ascoli spaces. Denote by $\tau_k^M$ the compact-open topology on $M_c(X)$.
\begin{proposition} \label{p:characterization-Ascoli}
Let $X$ be a Tychonoff space. Then:
\begin{enumerate}
\item[{\rm (i)}] $\tau_e\leq\tau_k^M$ on $M_c(X)$;
\item[{\rm (ii)}]  $\tau_e=\tau_k^M$ on $M_c(X)$ if and only if  $X$ is an Ascoli space.
\end{enumerate}
\end{proposition}
\begin{proof}
(i) immediately follows from Proposition \ref{p:Ascoli-Free-Ck}.

(ii) Assume that $X$ is an Ascoli space. By (i) we have to show that $\tau_k^M\leq\tau_e$. Take a standard $\tau_k^M$-neighborhood of zero
\[
[K;\e]=\{ \nu\in M_c(X): |\nu(f)|<\e \; \forall f\in K\},
\]
where $K$ is a compact subset of $\CC(X)$ and $\e>0$. Since $X$ is Ascoli, $K$ is equicontinuous and clearly  pointwise bounded. Therefore $[K;\e]$ is also a $\tau_e$-neighborhood of zero. Thus $\tau_k^M\leq\tau_e$.

Conversely, let $\tau_e=\tau_k^M$ on $M_c(X)$ and let $K$ be a compact subset of $\CC(X)$. Then the polar $K^\circ$ of $K$ is also a $\tau_e$-neighborhood of zero in $M_c(X)$. So there is an absolutely convex, equicontinuous and pointwise bounded subset $A$ of $C(X)$  such that $A^\circ\subseteq K^\circ$. By Proposition \ref{p:Ascoli-Free-Ck} we can assume that $A$ is pointwise closed. Now the Bipolar theorem implies that $K\subseteq A^{\circ\circ} =A$. So $K$ is equicontinuous. Thus $X$ is an Ascoli space.
\end{proof}

Recall that a locally convex space $E$ is called {\em semi-Montel} if every bounded subset of $E$ is relatively compact, and $E$ is a {\em Montel space} if it is a barrelled semi-Montel space.
For real semi-Montel spaces, the following result strengthens Proposition 2.4 of \cite{Gabr-free-resp}. 
\begin{theorem} \label{t:Schur-Montel-respected}
A real semi-Montel space $E$ respects all properties $\mathcal{P}\in\Pp$.
\end{theorem}

\begin{proof}
Let $A\in \mathcal{P}(E^+)$. Then $A$ is a functionally bounded subset of $E^+$. Hence $A$ is bounded in $E$ by Proposition \ref{p:FB-weakly-bound}. Therefore the closure $\overline{A}$ of $A$ in $E$ is compact and Proposition \ref{p:mu-space-Schur} applies. The last assertion follows from (\ref{equ:weak-Bohr-property}).
\end{proof}

Recall (see Theorem 15.2.4 of \cite{NaB}) that a locally convex space $E$ is semi-reflexive if and only if every bounded subset $A$ of $E$ is relatively weakly compact. In spite of the first assertion of the next corollary is known, see Corollary~4.15 of \cite{GalHer}, we give its simple and short proof.
\begin{corollary} \label{c:Schur-semi-reflexive}
Let $E$ be a real semi-reflexive lcs. Then $E$ has the  Glicksberg property if and only if $E$ is a semi-Montel space. In this case $E$ respects all properties $\mathcal{P}\in\Pp$.
\end{corollary}

\begin{proof}
Assume that $E$ has the  Glicksberg property. If $A$ is a  bounded  subset of $E$, then the weak closure $\overline{A}^{\, \tau_w}$ of $A$ is weakly compact, and hence $\overline{A}^{\, \tau_w}$ is compact also in $E$ by the Glicksberg property and (\ref{equ:weak-Bohr-property}). Thus $E$ is semi-Montel. The converse and the last assertions follow from Theorem \ref{t:Schur-Montel-respected}.
\end{proof}

Taking into account that any reflexive locally convex space is barrelled, Corollary \ref{c:Schur-semi-reflexive} immediately implies the main result of \cite{ReT}.
\begin{corollary}[\cite{ReT}] \label{c:Schur-reflexive}
Let $E$ be a real reflexive lcs. Then $E$ has the  Glicksberg property if and only if $E$ is a Montel space.
\end{corollary}
An example of a semi-Montel but non-Montel space is given in Corollary \ref{c:L(X)-semi-Montel} below.

\begin{proposition} \label{p:bounded-in-Mc}
Let $X$ be a  Dieudonn\'{e} complete space and let $\KK$ be a $\tau_e$-closed subset of $M_c(X)$. Then the following assertions are equivalent:
\begin{enumerate}
\item[{\rm (i)}] $\KK$ is $\tau_e$-compact;
\item[{\rm (ii)}] $\KK$ is $\tau_e$-bounded;
\item[{\rm (iii)}] there is a compact subset $C$ of $X$ and $\e>0$ such that $\KK\subseteq [C;\e]^\circ$.
\end{enumerate}
In particular, the space $(M_c(X),\tau_e)$ is a semi-Montel space.
\end{proposition}
\begin{proof}
(i)$\Rightarrow$(ii) is clear. Let us prove that (ii)$\Rightarrow$(iii). Since $X$ being a  Dieudonn\'{e} complete  space is a $\mu$-space, $\CC(X)$ is barrelled by the Nachbin--Shirota theorem. This fact and Corollary \ref{p:Ck-Mc-compatible} imply that $\KK$ is equicontinuous. So there is a compact subset $C$ of $X$ and $\e>0$ such that $\KK\subseteq [C;\e]^\circ$. To prove (iii)$\Rightarrow$(i) we note first that $[C;\e]^\circ$ is equicontinuous and $\sigma\big(M_c(X),C(X)\big)$-compact by the Alaoglu theorem. Therefore $[C;\e]^\circ$ is compact in the precompact-open topology $\tau_{pc}$ on $M_c(X)$ by Proposition 3.9.8 of \cite{horvath}.  By Proposition \ref{p:characterization-Ascoli}, we have $\tau_e\leq \tau_k^M \leq \tau_{pc}$. Hence $[C;\e]^\circ$ is $\tau_e$-compact. Thus $\KK$ being closed is also $\tau_e$-compact.
\end{proof}

Theorem \ref{t:Schur-Montel-respected} and Proposition \ref{p:bounded-in-Mc} imply
\begin{corollary} \label{c:Mc-respect}
If $X$ is a  Dieudonn\'{e} complete space, then $(M_c(X),\tau_e)$ respects all properties $\mathcal{P}\in\Pp$.
\end{corollary}

Below we describe bounded subsets of $L(X)$, this result  generalizes Lemma 6.3 of \cite{GM}. For  $\chi = a_1 x_1+\cdots +a_n x_n\in L(X)$ with distinct $x_1,\dots, x_n\in X$ and  nonzero $a_1,\dots,a_n\in\IR$, we set
\[
\| \chi\|:=|a_1|+\cdots +|a_n|, \; \mbox{ and } \; \mathrm{supp}(\chi):=\{ x_1,\dots, x_n\},
\]
and recall that
\[
f(\chi)=a_1 f(x_1)+\cdots +a_n f(x_n), \; \mbox{ for every } f\in C(X)=L(X)'.
\]
For $\{ 0\}\not= A\subseteq L(X)$, set $\mathrm{supp}(A):=\bigcup_{\chi\in A} \mathrm{supp}(\chi)$.

\begin{proposition} \label{p:bounded-in-L(X)}
A nonzero subset $A$ of $L(X)$ is bounded if and only if $\supp(A)$ has compact closure in the Dieudonn\'{e} completion $\mu X$ of $X$ and $C_A:=\sup\{ \| \chi\|: \chi\in A\}$ is finite.
\end{proposition}

\begin{proof}
Observe that a subset $B$ of an lcs $E$ is bounded if and only if its closure $\overline{B}$ in the completion $\overline{E}$ of $E$ is bounded. Now assume that $A$ is bounded. By Theorem \ref{t:Free-complete-L}, we have $\overline{L(X)}=(M_c(\mu X),\tau_e)$ and, by Corollary \ref{p:Ck-Mc-compatible}, the topology $\tau_e$ is compatible with the duality $(\CC(\mu X),M_c(\mu X))$. As $\mu X$ is a $\mu$-space, the Nachbin--Shirota theorem implies that $\CC(\mu X)$ is barrelled. Therefore $A$ is a bounded subset of $L(X)$ if and only if its completion $\overline{A}$ in $(M_c(\mu X),\tau_e)$ is equicontinuous and hence if and only if there is a compact subset $K$ of $\mu X$ and $\e>0$ such that $A\subseteq [K;\e]^\circ \cap L(X)$. By the regularity of $\mu X$ it is easy to see that
\[
\chi = a_1 x_1+\cdots +a_n x_n \in [K;\e]^\circ \cap L(X),
\]
where $x_1,\dots, x_n\in X$ are distinct and $a_1,\dots,a_n$ are nonzero, if and only if $x_1,\dots,x_n\in K$ and $\| \chi\|=|a_1|+\cdots +|a_n|\leq 1/\e$. Therefore, if $A$ is bounded, then $\supp(A)\subseteq K$ and $C_A <1/\e$.

Conversely, let $\overline{\supp(A)}$ be compact in $\mu X$ and $C_A<\infty$. Set $B:= \big[\overline{\supp(A)}; 1/C_A \big]^\circ$. Then $B$ is $\sigma\big(M_c(\mu X),C(\mu X)\big)$-compact by the Alaoglu theorem. Therefore $B$ is compact in the precompact-open topology $\tau_{pc}$ on $M_c(\mu X)$ by Proposition 3.9.8 of \cite{horvath}. Since $\tau_e\leq \tau_k^M \leq \tau_{pc}$ by Proposition \ref{p:characterization-Ascoli}, we obtain that $B$ is a $\tau_e$-compact subset of $M_c(\mu X)$. As $A\subseteq B\cap L(X)$, the above observation implies that $A$ is a bounded subset of $L(X)$.
\end{proof}





\begin{corollary} \label{c:L(X)-semi-Montel}
Let $X$ be a Dieudonn\'{e} complete space whose compact subsets are finite. Then $L(X)$ is a complete semi-Montel space. If, in addition, $X$ is non-discrete, then $L(X)$ is not Montel.
\end{corollary}

\begin{proof}
By Proposition \ref{p:bounded-in-L(X)}, every bounded subset $A$ of $L(X)$ is a bounded subset of a finite-dimensional subspace of $L(X)$. Therefore $\overline{A}$ is compact and hence $L(X)$ is a semi-Montel space. The space $L(X)$ is complete by Theorem \ref{t:Free-complete-L}. If additionally $X$ is not discrete, then $L(X)$ is not barrelled  by Theorem 6.4 of \cite{GM}. 
Thus $L(X)$ is not a Montel space.
\end{proof}

Now we prove the main result of this section.

\begin{proof}[Proof of Theorem \ref{t:L(X)-Schur}]
By Theorem \ref{t:Free-complete-L}, the space  $L(X)$ embeds into $(M_c(\mu X),\tau_e)$. Now Proposition \ref{pSubSchur} and Corollary \ref{c:Mc-respect} imply that $L(X)$ respects all properties $\mathcal{P}\in\Pp_0$.

Assume in addition that $L(X)$ is complete. Then $X$  is Dieudonn\'{e} complete and does not have infinite compact subsets by Theorem \ref{t:Free-complete-L}.
Hence $L(X)$ is a  semi-Montel space by Corollary \ref{c:L(X)-semi-Montel}. Thus $L(X)$ respects also functional boundedness by Theorem \ref{t:Schur-Montel-respected}.
\end{proof}


\smallskip
\begin{proof}[Proof of Corollary \ref{c:L(X)-quotients}]
By the universal property of the free lcs $L(E)$, the identity map $id: E\to E$ extends to a continuous linear map $\overline{id}$ from $L(E)$ onto $E$. Since $E$ is a subspace of $L(E)$ it follows that $\overline{id}$ is a quotient map. It remains to note that $L(E)$ weakly respects all properties $\mathcal{P}\in \Pp_0$ by Theorem \ref{t:L(X)-Schur}.
\end{proof}

Following \cite{Mar}, an abelian topological group $A(X)$ is called the {\em  free abelian topological  group} over  a Tychonoff space $X$ if there is a continuous map  $i: X\to A(X)$ such that $i(X)$ algebraically generates $A(X)$, and if $f: X\to G$ is a continuous map to an abelian topological  group $G$, then there exists a continuous homomorphism ${\bar f}: A(X) \to G$  such that $f={\bar f} \circ i$. The free abelian topological  group $A(X)$ is always exists and is essentially unique. The identity map $id_X :X\to X$ extends to a canonical homomorphism $id_{A(X)}: A(X)\to L(X)$. Note that $id_{A(X)}$ is an embedding of topological groups, see \cite{Tkac,Usp2}.

It is known (see \cite{GalHer}) that the free abelian topological group $A(X)$  over a Tychonoff space $X$ has the Glicksberg property. The next corollary generalizes this result.
\begin{corollary} \label{c:A(X)-Schur}
Let $X$ be a Tychonoff space. Then the free abelian topological group  $A(X)$ over $X$ is locally quasi-convex and respects all properties $\mathcal{P}\in\Pp_0$.
\end{corollary}

\begin{proof}
Since $A(X)$ is a subgroup of $L(X)$, $A(X)$ is locally quasi-convex. The group $A(X)$ respects all properties $\mathcal{P}\in\Pp_0$ by Proposition \ref{pSubSchur} and  Theorem \ref{t:L(X)-Schur}.
\end{proof}

\begin{proof}[Proof of Corollary \ref{c:A(X)-quotients}]
Note that $G$ is a Tychonoff space. So, by the universal property of $A(G)$, the identity map $id: G\to G$ extends to a continuous homomorphism $\overline{id}$ from $A(G)$ onto $G$. Clearly, $\overline{id}$ is a quotient map. Now Corollary \ref{c:A(X)-Schur} finishes the proof.
\end{proof}

We do not know whether $L(X)$ and $A(X)$ respect also functional boundedness for every Tychonoff space $X$.

Below we give an application of the obtained results.
Let $X$ be a  Tychonoff space and let $G$ be an abelian locally compact group. Denote by $\CC(X,G)$ the space $C(X,G)$ endowed with the compact-open topology. Then $\CC(X,G)$ is an abelian topological group under the pointwise addition. Pol and Smentek \cite{PolSmen} proved that the group $\CC(X,D)$ is reflexive for every  finitely generated discrete abelian group $D$ and each zero-dimensional realcompact $k$-space $X$. Au\ss enhofer proved in \cite[Theorem~14.9]{Aus} that the group $\CC(X,\Ss)$ is reflexive for every hemicompact $k$-space $X$.
Let $\mathbf{s}=\{ e_n\}_{n\in\NN} \cup \{e_0\}$ be a one-to-one convergent sequence with the limit point $e_0$.
It is easy to see that the map
\[
F:\CC(\mathbf{s},G)\to G\times \Ff_0(G), \quad F(h):=\Big( h(e_0); \big( h(e_n)-h(e_0) \big)_{n\in\NN} \Big), \; h\in \CC(\mathbf{s},G),
\]
is a topological isomorphism. Since the group $\Ff_0(G)$ is reflexive by \cite{Ga8}, we obtain that the group $\CC(\mathbf{s},G)$ is reflexive for every abelian locally compact group $G$. These results suggest the following question:
{\em For an abelian locally compact group  $G$, characterize Tychonoff spaces $X$ such that  the group $\CC(X,G)$ is reflexive.}
For the most important case $G=\IR$ we obtain a partial answer to this question.
\begin{proposition} \label{p:reflexive-C(X,G)}
If $X$ is a Dieudonn\'{e} complete Ascoli space, then the space $\CC(X)$ is a reflexive group.
\end{proposition}

\begin{proof}
Let $\alpha: \CC(X)\to\CC(X)^{\wedge\wedge}$ be the canonical map.
Since $X$ is Ascoli, Fact \ref{f:dual-E} and Proposition  \ref{p:characterization-Ascoli} imply that $(M_c(X),\tau_e)$ is topologically isomorphic to $\CC(X)^\wedge$ under the canonical map $\psi(\chi):= e^{2\pi i \chi}, \chi\in M_c(X)$. Applying Fact \ref{f:dual-E} and Corollary \ref{p:Ck-Mc-compatible} we obtain that $\alpha$ is surjective. As $\CC(X)$ is a locally quasi-convex group by Proposition 2.4 of \cite{Ban}, $\alpha$ is open and injective by \cite[Proposition~6.10]{Aus}. To show that $\alpha$ also is continuous it is sufficient to prove that every compact subset of $\CC(X)^\wedge$ is equicontinuous, see  \cite[Proposition~5.10]{Aus}. Recall that every Dieudonn\'{e} complete space is a $\mu$-space (see \cite[Proposition~1.19]{Aus}), and hence $\CC(X)$ is barrelled by the Nachbin--Shirota theorem. As,  by Corollary \ref{p:Ck-Mc-compatible}, the topology $\tau_e$ is compatible with the weak-$\ast$ topology on $M_c(X)$, we obtain that all compact subset of $(M_c(X),\tau_e)$ are equicontinuous. Hence all compact subsets of $\CC(X)^\wedge$ are equicontinuous by (vi) of Proposition \ref{p:weak*-space-group}. Thus $\alpha$ is a topological isomorphism.
\end{proof}
In particular, $\CC(X)$ is a reflexive group for every separable metrizable space $X$. It is worth mentioning that $\CC(X)$ is a reflexive space if and only if $X$ is discrete, see Theorem 11.7.7 of \cite{Jar}.

We end this section with the following questions.
\begin{question}
Characterize Tychonoff spaces $X$ for which $C_p(X)$ and $L(X)$ are reflexive groups. Does there exist a non-discrete $X$ such that $C_p(X)$ or $L(X)$ is a reflexive group?
\end{question}





\section{$\mathcal{P}$-barreledness, reflexivity and respecting properties} \label{sec:3} 


Let $E$ be a locally convex space.
It is well-known that $E$ is {\em barrelled} if and only if every $\sigma(E',E)$-bounded subset of $E'$ is equicontinuous.
Recall that $E$ is called {\em $c_0$-barrelled} if every $\sigma(E',E)$-null sequence is equicontinuous. Analogously, if $\mathcal{P}$ is a (topological) property, we shall say that  $E$ is a {\em $\mathcal{P}$-barreled space} if every $A\in \mathcal{P}\big(E',\sigma(E',E)\big)$ is equicontinuous.

Following \cite{CMPT}, a $MAP$ abelian group $G$ is called {\em $g$-barrelled} if any $\sigma(\widehat{G},G)$-compact subset of  $\widehat{G}$ is equicontinuous. Every real barrelled lcs $E$ is a $g$-barrelled group, but the converse does not hold in general by \cite{CMPT} (see also Example \ref{exa:g-barreled-nobarreled} below). Analogously, $G$  is {\em sequentially barrelled} or {\em $c_0$-barrelled} if  any $\sigma(\widehat{G},G)$-convergent sequence of  $\widehat{G}$ is equicontinuous, see \cite{MT}. More generally, for a property $\mathcal{P}$, we shall say that a $MAP$ abelian group  $G$ is {\em $\mathcal{P}$-barrelled} if every $A\in \mathcal{P}\big(\widehat{G},\sigma(\widehat{G},G)\big)$ is equicontinuous. Clearly, every $g$-barrelled group is also $c_0$-barrelled. In the next proposition item (i) extends Proposition 1.12 of \cite{CMPT} and explains our use of the notion ``$c_0$-barrelled group'' also for $c_0$-barrelled spaces.
\begin{proposition} \label{p:space-group-barrelled}
Let $E$ be a real locally convex space and let $\mathcal{P}\in\Pp_0$. Then:
\begin{enumerate}
\item[{\rm (i)}] $E$ is a $\mathcal{P}$-barreled space if and only if $E$ is a $\mathcal{P}$-barrelled group;
\item[{\rm (ii)}] if $E$ is a barrelled space, then $E$ is a $\mathcal{P}$-barreled group.
\end{enumerate}
\end{proposition}

\begin{proof}
(i) Let $E$ be a $\mathcal{P}$-barreled space and  $A\in \mathcal{P}\big(\widehat{E},\sigma(\widehat{E},E)\big)$.  Recall that the canonical isomorphism $\psi: E' \to \widehat{E}$ is defined by $\psi(\chi):= e^{2\pi i \chi}$. Then $\psi^{-1}(A)\in \mathcal{P}\big(E',\sigma(E',E)\big)$ by (ii) of Proposition \ref{p:weak*-space-group}, and hence $\psi^{-1}(A)$ is equicontinuous. Now (vi) of Proposition \ref{p:weak*-space-group} implies that $A$ is equicontinuous. Thus $E$ is a $\mathcal{P}$-barreled group.
Conversely, let $E$ be a $\mathcal{P}$-barreled group and  $A\in \mathcal{P}\big(E',\sigma(E',E)\big)$. Then $\psi(A)\in \mathcal{P}\big(\widehat{E},\sigma(\widehat{E},E)\big)$ by (ii) of Proposition \ref{p:weak*-space-group}, and hence $\psi(A)$ is equicontinuous. Applying now (vi) of Proposition \ref{p:weak*-space-group} we obtain that $A$ is equicontinuous. Thus $E$ is a $\mathcal{P}$-barreled space.

(ii) follows from (i) and the fact that every weak-$\ast$ functionally bounded subset of $E'$ is weak-$\ast$ bounded, see Proposition \ref{p:FB-weakly-bound}.
\end{proof}

In what follows, for $\mathcal{P}=\mathcal{C}$ and  $\mathcal{P}=\mathcal{S}$, we shall use the standard terminology of being $g$-barrelled or $c_0$-barrelled  instead of  being  $\mathcal{C}$-barrelled or $\mathcal{S}$-barrelled, respectively.

Items (i) and (ii) of the next proposition generalizes $(a)$-$(b')$ of Proposition 2.3 of \cite{MT}.

\begin{proposition} \label{p:Mackey-g-barrelled}
Let $G$ be a $MAP$ abelian group. Then:
\begin{enumerate}
\item[{\rm (i)}] if $G$ is $\mathcal{P}$-barrelled for $\mathcal{P}\in\Pp_0$, then $G^\wedge$ respects $\mathcal{P}$;
\item[{\rm (ii)}] if $G^\wedge$ is $\mathcal{P}$-barrelled for $\mathcal{P}\in\Pp_0$ and $\alpha_G$ is a topological embedding, then $G$ respects $\mathcal{P}$;
\item[{\rm (iii)}] if $G$ is reflexive, then $G$ is $c_0$-barrelled if and only if $G^\wedge$ has the Schur property;
\item[{\rm (iv)}]  if $G$ is reflexive, then $G$ is $g$-barrelled if and only if $G^\wedge$ has the Glicksberg property;
\item[{\rm (v)}] if $G$ is reflexive and $G^\wedge$ is angelic, then $G$ is $\mathcal{CC}$-barrelled if and only if $G^\wedge$ respects $\mathcal{CC}$.
\end{enumerate}
\end{proposition}
\begin{proof}
(i) Let $A\in \mathcal{P}\big( \widehat{G},\sigma(\widehat{G},G^{\wedge\wedge})\big)$. Then $A\in \mathcal{P}\big( \widehat{G},\sigma(\widehat{G},G)\big)$ as well, so $A$ is equicontinuous. Hence there is a neighborhood $U$ of zero in $G$ such that $A\subseteq U^\triangleright$ and the set $U^\triangleright$ is a compact subset of $G^\wedge$, see Fact \ref{f:Mackey-Ban}. So the identity map
$
\big( U^\triangleright, \tau_k|_{U^\triangleright}\big) \mapsto \big( U^\triangleright, \sigma(\widehat{G},G^{\wedge\wedge})|_{U^\triangleright}\big)
$
is a homeomorphism, where $\tau_k$ is the compact-open topology of the dual group $G^\wedge$. Therefore $A\in \mathcal{P}(G^\wedge)$, and hence $G^\wedge$ respects $\mathcal{P}$.

(ii) follows from (i) and Proposition \ref{pSubSchur}.

(iii) Assume that $G^\wedge$ has the Schur property and $S$ is  a $\sigma(\widehat{G},G)$-null sequence in $\widehat{G}$. By the reflexivity of $G$, $S$ is also a $\sigma(\widehat{G},G^{\wedge\wedge})$-null sequence. Hence $S$ converges to zero in $G^\wedge$ by the Schur property. Therefore $S^\triangleright$ is a neighborhood of zero in $G^{\wedge\wedge}$. So, by the reflexivity of $G$, $S^\triangleleft =\alpha_G^{-1}(S^\triangleright)$ is a neighborhood of zero in $G$. Since $S\subseteq S^{\triangleleft\triangleright} $ we obtain that $S$ is equicontinuous, see Fact \ref{f:Mackey-Ban}. Thus $G$ is $c_0$-barrelled. The converse assertion follows from (i).

The proof of (iv) is similar to the proof of (iii).

(v) Assume that $G^\wedge$ respects countable compactness. Let $A$ be  a $\sigma(\widehat{G},G)$-countably compact subset of  $\widehat{G}$. By the reflexivity of $G$, $A$ is also $\sigma(\widehat{G},G^{\wedge\wedge})$-countably compact. As $G^\wedge$ respects $\mathcal{CC}$, $A$ is countably compact in $G^\wedge$ and hence $A$ is compact by the angelicity of $G^\wedge$. The rest of the proof repeats (iii) replacing $S$ by $A$.
%
%
\end{proof}

\begin{corollary}[\cite{CMPT}]
A locally compact abelian group $G$ is $g$-barrelled.
\end{corollary}
\begin{proof}
Since $G$ is reflexive and $G^\wedge$ is locally compact by the Pontryagin--van Kampen duality theorem, the assertion follows from Glicksberg's theorem and (iv) of Proposition \ref{p:Mackey-g-barrelled}.
\end{proof}

The condition of being reflexive in Proposition \ref{p:Mackey-g-barrelled} is essential as the following example shows. 

\begin{example} \label{exa:non-reflexive-g-barrelled}
The  group $G:=C_p(\mathbf{s},2)$ of all continuous maps from $\mathbf{s}$ to the discrete group $\ZZ(2)$ with the pointwise topology has the following properties:
\begin{enumerate}
\item[{\rm (i)}] $G$ is a countable non-reflexive precompact metrizable group;
\item[{\rm (ii)}] every compact subset of $G^\wedge$ is equicontinuous;
\item[{\rm (iii)}] $G^\wedge$ respect all the properties $\mathcal{P}\in \Pp$;
\item[{\rm (iv)}]  $G$ is not $c_0$-barrelled.
\end{enumerate}
\end{example}

\begin{proof}
(i) Observe that $G$ is a dense proper subgroup of the compact metrizable group $\ZZ(2)^\NN$, so $G$ is metrizable. Being non-complete $G$ is not reflexive, see \cite{Cha}. To show that $G$ is countable, for every $n\in\NN$, set $F_n:=\{ e_1,\dots,e_n\}$ and $U_n:= \mathbf{s}\setminus F_n$. If $f\in G$, there is an $n\in\NN$ such that $f|_{U_n}= f(e_0)\in \ZZ(2)$. Therefore $f$ is uniquely defined by its values on $F_n\cup \{e_0\}$. Thus $G$ is countable.

(ii),(iii) By \cite{Aus,Cha}, the group $G^\wedge$ is the countable direct sum $\bigoplus_\NN \ZZ(2)$ endowed with the discrete topology. So every compact subset of $G^\wedge$ is finite and hence equicontinuous. Since  $G^\wedge$ is discrete, it respects all properties $\mathcal{P}\in \Pp$ (see Introduction). 

(iv) For every $n\in\NN$, set
\[
\chi_n := \big( \underbrace{0,\dots,0}_{2n}, 1,1,0, \dots\big)
\]
Taking into account the description of continuous functions given in (i), we obtain
\[
\chi_n(f)=\exp\big\{ \pi i \big( f(e_{2n+1}) + f(e_{2n+2})\big) \big\} =1
\]
for all sufficiently large $n\in\NN$. Thus $\chi_n\to 0$ in the pointwise topology on $\widehat{G}$. To show that $G$ is not $c_0$-barrelled it suffices to prove that the sequence $S:=\{ \chi_n\}_{n\in\NN}$ is not equicontinuous. For every $n\in\NN$, define $f_n\in G$ by
\[
f_n(e_n):=1, \mbox{ and } f_n(e_m)=0 \mbox{ if } m\not= n.
\]
It is clear that $f_n \to 0$ in $G$. Since $\chi_n(f_{2n+1})=\exp\{\pi i\} =-1$ we obtain that $S$ is not equicontinuous.
\end{proof}

In Remark~16 of \cite{CMPT}, Mendoza has pointed out that if $E$ is a non-reflexive real Banach space, then $\big(E',\mu(E',E)\big)$, where $\mu(E',E)$ is the Mackey topology on $E'$, is a $g$-barrelled lcs which is not barrelled. 
So the converse in (ii) of Proposition \ref{p:space-group-barrelled} is not true in general. Below we propose an analogous example of a $g$-barrelled real lcs $E$ which is not barrelled.
\begin{example} \label{exa:g-barreled-nobarreled} {\em
Let $(E,\tau)$ be a real non-semi-reflexive lcs. Assume that $E$ is complete and has the Glicksberg property (for example, $E=\ell_1^\kappa$ for some cardinal $\kappa$). 
Set $F:=\big(E',\mu(E',E)\big)$, where $\mu(E',E)$ is the Mackey topology on $E'$. As $E$ is not semi-reflexive, the space $F$ is not barrelled by Theorem 11.4.1 of \cite{Jar}.
To show that $F$ is a $g$-barrelled space take arbitrarily a compact subset $K$ of $\big(F',\sigma(F',F)\big)=\big(E,\sigma(E,E')\big)$. Denote by $C:=\overline{\mathrm{acx}}(K)$ the closed absolutely convex hull of $K$. We claim that $C$ is also $\sigma(E,E')$-compact. Indeed, the set $K\cup(-K)$ is $\tau$-compact by the Glicksberg property of $E$. So $C$ is $\tau$-compact in $E$ by Theorem 4.8.9 of \cite{NaB}. Thus $C$ is $\sigma(E,E')$-compact as well. 
Now the definition of the Mackey topology $\mu(E',E)$ and the claim imply that $C$ and hence $K$ are equicontinuous. Therefore $F$ is a $g$-barrelled space.

Assume additionally that $E$ is a Banach space. Then $E$ is a reflexive group by \cite{Smith}, and Proposition \ref{p:Mackey-g-barrelled}(iv) implies that $E^\wedge$ is a $g$-barrelled group. Hence $(E^\wedge)^\wedge =E$ and the group $E^\wedge$ is a Mackey group, see \cite{CMPT} (or Proposition \ref{p:g-barrelled-Mackey} below). By Fact \ref{f:dual-E}, the dual space $E'$ endowed with the compact-open topology $\tau_k$ is topologically isomorphic to $E^\wedge$. Therefore $(E',\tau_k)$ is a Mackey space such that $(E',\tau_k)'=E$. Thus $\tau_k=\mu(E',E)$ by the uniqueness of the Mackey space topology.
$\Box$ }
%
%
\end{example}

Every abelian locally quasi-convex  $k_\w$-group $G$ is a Schwartz group by Corollary 5.5 of \cite{ACDT}, and hence $G$ has the Glicksberg property by \cite{Aus2}. Below we essentially generalize this result using a completely different method.
\begin{theorem} \label{t:k-omega-Schur}
An abelian locally quasi-convex $k_\w$-group $G$  respects all properties $\mathcal{P}\in \Pp$.
\end{theorem}

\begin{proof}
First we prove the following claim.

{\em Claim. If $G$ is a metrizable abelian group, then $G^\wedge$  respects all properties $\mathcal{P}\in \Pp$.} Indeed, let $\bar{G}$ be the completion of $G$. Then $\bar{G}^\wedge =G^\wedge$ by \cite{Aus,Cha} and $\bar{G}$ is $g$-barrelled by Corollary 1.6 of \cite{CMPT}. Therefore $G^\wedge$ has the Glicksberg property by (i) of Proposition \ref{p:Mackey-g-barrelled}. By \cite{Aus,Cha}, $G^\wedge$ is a $k_\w$-space, and hence $(G^\wedge)^+$ is a $\mu$-space. Since every $k_\w$-group is complete by \cite{HuntMorris}, the group  $G^\wedge$  respects all properties $\mathcal{P}\in \Pp$ by Theorem 1.2 of \cite{Gab-Top-Nul}. The claim is proved.

Note that $G^\wedge$ is metrizable, and hence $G^{\wedge\wedge}$ respects all properties $\mathcal{P}\in \Pp$ by the claim. By 5.12 and 6.10 of \cite{Aus}, the canonical homomorphism $\alpha_G$ is an embedding of $G$ into $G^{\wedge\wedge}$. Since $G$ is complete by \cite{HuntMorris},  $\alpha_G(G)$ is a closed subgroup of the $k_\w$-group $G^{\wedge\wedge}$. Therefore $\alpha_G(G)$ is $C$-embedded in $G^{\wedge\wedge}$ by \cite[3D.1]{GiJ}. Thus $G$ respects all properties $\mathcal{P}\in \Pp$ by Propositions  4.9 of  \cite{Gab-Top-Nul} and Proposition  \ref{pSubSchur}.
\end{proof}

Following \cite{GGH}, a topological group $G$ is called a {\em locally $k_\w$-group} if it has an open $k_\w$-subgroup.
\begin{corollary} \label{c:locally-k-omega}
A locally quasi-convex locally $k_\w$-group $G$  respects all properties $\mathcal{P}\in \Pp_0$.
\end{corollary}

\begin{proof}
The assertion follows from Theorem \ref{t:k-omega-Schur} and Proposition  \ref{pSubSchur}.
\end{proof}

In the next example we show that the condition of being locally quasi-convex cannot be dropped in Theorem \ref{t:k-omega-Schur}.  Denote by $\VV(\mathbf{s})$ and $L(\mathbf{s})$ the free topological vector space and the free locally convex space over the convergent sequence $\mathbf{s}$, respectively.
\begin{example} \label{exa:k-omega-not-Schur}
{\rm (i)}  $\VV(\mathbf{s})$ is a non-locally quasi-convex $MAP$ $k_\w$-group, so $\VV(\mathbf{s})$ is a Schwartz group;

{\rm (ii)}  $\VV(\mathbf{s})$ does not have the Schur property, and hence $\VV(\mathbf{s})$ does not respect any  $\mathcal{P}\in \Pp$.
\end{example}

\begin{proof}
(i) The space $\VV(\mathbf{s})$ is a $k_\w$-group by Theorem 3.1 of \cite{GM}, and it is not locally quasi-convex by Proposition 5.13 of \cite{GM} and the fact that a topological vector space $E$ is a locally quasi-convex group if and only if $E$ is locally convex (see Proposition 2.4 of \cite{Ban}). By Proposition 5.1 of \cite{GM}, the space $\VV(\mathbf{s})$ is a $MAP$ group. As a $k_\w$-group, $\VV(\mathbf{s})$ is a Schwartz group by Corollary 5.5 of \cite{ACDT}.

(ii) Since the spaces $\VV(\mathbf{s})$ and $L(\mathbf{s})$ have the same dual space (see Proposition 5.10 of \cite{GM}), it is sufficient to find a sequence $\{ z_k\}_{k\in\NN}$ such that $\{ z_k\}_{k\in\NN}$ converges in $L(\mathbf{s})$ but it does not converge in $\VV(\mathbf{s})$.  For every $k\in\NN$, set $d_k :=2^k$ and put
\[
z_k :=\frac{1}{d_{k+1}-d_k} \left( e_{d_{k}+1} +\cdots + e_{d_{k+1}}\right).
\]
Then $z_k\to e_0$ in $L(\mathbf{s})$ because $L(\mathbf{s})$ is locally convex and since $e_n \to e_0$. On the other hand, $z_k\not\to e_0$ in $\VV(\mathbf{s})$ by Corollary 3.4 of \cite{GM}.

Since any $k_\w$-space is a $\mu$-space, the last assertion follows from Proposition \ref{p:Schur=sequential-comp}.
\end{proof}

Recall that a topological group $X$ is said to have a {\em subgroup topology} if it has a base at the identity consisting of subgroups.
In the next section we use the following proposition.
\begin{proposition} \label{p:nuclear-Mackey}
{\em (i)} If $G$ is an abelian topological group with a subgroup topology, then $G$ is a  locally quasi-convex nuclear group. So $G$ respects all properties $\mathcal{P}\in \Pp_0$.

{\em (ii)} If $(G,\tau)$ is a locally quasi-convex abelian group of finite exponent, then $(G,\tau)$ and hence also $(G,\tau)^\wedge$ respect all the properties $\mathcal{P}\in \Pp_0$.
\end{proposition}

\begin{proof}
(i) By Proposition 2.2 of \cite{AG}, $G$ embeds into a product of discrete groups. Therefore $G$ is a  locally quasi-convex nuclear group by Propositions 7.5 and 7.6 and Theorem 8.5 of \cite{Ban}. Finally, the group $G$ respects all properties $\mathcal{P}\in \Pp_0$ by Corollary 4.7 of  \cite{Gab-Top-Nul}.

(ii) Propositions 2.1  of \cite{AG} implies that the topologies of the groups $(G,\tau)$ and  $(G,\tau)^\wedge$ are subgroup topologies, and (i) applies.
\end{proof}


Being motivated by \cite{MT}, we consider below a ``compact'' version of the Dunford--Pettis property for abelian topological groups.
Let $X$ and $Y$ be topological spaces. A map $p:X\to Y$ is called {\em $s$-continuous} ({\em $k$-continuous}) if the restriction of $p$ onto every convergent sequence (every compact subset, respectively) of $X$ is continuous. Clearly, every $k$-continuous map is also $s$-continuous.
The next lemma is straightforward.
\begin{lemma} \label{l:evaluation-k-cont}
For every abelian topological group $G$ the evaluation map $\psi: G^\wedge \times G \to \Ss$, $\psi(\chi,g):=\chi(g)$, is $k$-continuous.
\end{lemma}


Following \cite{MT}, an abelian topological group $G$ has the {\em sequential Bohr continuity property} ($s$-$BCP$, for short) if the map
\begin{equation} \label{equ:eval-map-Bohr}
\psi: \big( \widehat{G}, \sigma(\widehat{G}, G^{\wedge\wedge})\big) \times \big( G,\sigma(G,G^\wedge)\big) \to \Ss, \quad \psi(\chi,g):=\chi(g),
\end{equation}
is $s$-continuous. We shall say that $G$ has the {\em $k$-Bohr continuity property} ($k$-$BCP$, for short) if the map in (\ref{equ:eval-map-Bohr}) is $k$-continuous. Clearly, if $G$ has the $k$-$BCP$ then it has also the $s$-$BCP$. The next assertion is an analogue of Proposition 2.4 of \cite{MT} and generalizes (d) and (c) of this proposition.
\begin{proposition} \label{p:k-BCP}
For an abelian topological group $G$ the following assertionds hold:
\begin{enumerate}
\item[{\rm (i)}] if $G$ and $G^\wedge$ have the Glicksberg property, then $G$ has the $k$-$BCP$;
\item[{\rm (ii)}] if $G$ has the Glicksberg property and is $g$-barrelled, then $G$ has the $k$-$BCP$;
\item[{\rm (iii)}] if $G$ is metrizable and has the Glicksberg property, then $G$ has the $k$-$BCP$;
\item[{\rm (iv)}] if $G$ is locally compact, then $G$ has the $k$-$BCP$;
\item[{\rm (v)}] if $G$ is a locally quasi-convex almost metrizable Schwartz group, then $G$ has the $k$-$BCP$;
\item[{\rm (vi)}] if $G$ is reflexive, then $G$ has the $k$-$BCP$ if and only if $G^\wedge$ has the $k$-$BCP$.
\end{enumerate}
\end{proposition}

\begin{proof}
(i) immediately follows from Lemma \ref{l:evaluation-k-cont}, and (ii) follows from (i) and Proposition \ref{p:Mackey-g-barrelled}. The Claim of Theorem \ref{t:k-omega-Schur} and (i) imply (iii). (iv) follows from (i) and the Glicksberg theorem.

(v) By \cite{Aus2}, $G$ has the Glicksberg property. The dual group $G^\wedge$ of $G$ is a locally quasi-convex locally $k_\w$-group by Proposition 7.1 of \cite{GGH}. Therefore $G^\wedge$ has the Glicksberg property by Corollary \ref{c:locally-k-omega}. Now the assertion follows from (i).

(vi) follows from the definition of the evaluation map $\psi$ and the reflexivity of $G$.
\end{proof}

Let $G$ and $H$ be abelian topological groups. Denote by $\mathrm{CHom}_p(G,H)$ the group of all continuous homomorphisms from $G$ to $H$ endowed with the pointwise topology.
\begin{proposition} \label{p:CHom-Schur}
Let $G$ and $H$ be $MAP$ abelian groups with the Schur property. Then also the group $Z=\mathrm{CHom}_p(G,H)$ has the Schur property.
\end{proposition}

\begin{proof}
For every $\Phi\in \widehat{H}$ and each $g\in G$, define $\Phi^\ast_g :Z\to \Ss$ by $\Phi^\ast_g (\chi):= \Phi\big( \chi(g)\big)$ ($\chi\in Z$). We claim that the homomorphism $\Phi^\ast_g $ is continuous, i.e. $\Phi^\ast_g \in\widehat{Z}$. Indeed, fix $\e>0$. Choose an open neighborhood $V$ of zero in $H$ such that $|\Phi(h)-1|<\e$ for every $h\in V$. Set $\Delta := \{ -g,0,g\}$. Now if
\[
\chi \in [\Delta;V] :=\{ \chi\in Z: \chi(t)\in V \;\; \forall t\in \Delta\},
\]
then $|\Phi^\ast_g (\chi)-1|=| \Phi\big( \chi(g)\big)-1|<\e$. Thus $\Phi^\ast_g$ is continuous.

Now suppose for a contradiction that $Z$ does not have the Schur property. Then there exists a $\sigma(Z,\widehat{Z})$-null sequence $\{\chi_n\}_{n\in\NN}$ in $Z$ which does not converge to zero in $Z$. So there is a finite subset $F$ of $G$, an open neighborhood $U$ of zero in $H$ and $\e>0$ such that
\begin{equation} \label{equ:Hom-p-1}
\chi_n \not\in [F;U] :=\{ \chi\in Z: \chi(g)\in U \;\; \forall g\in F\},
\end{equation}
for infinitely many indices $n$. Passing to a subsequence if needed, we shall assume that (\ref{equ:Hom-p-1}) holds for all $n\in\NN$. 
Since $F$ is finite, we can assume that $F=\{ g\}$ for some $g\in G$. Therefore $\chi_n(g)\not\in U$ for every $n\in\NN$. But this means that the sequence $\chi_n(g)$ is not a null sequence in $H$. By the Schur property  of $H$, there are $\Phi\in\widehat{H}$, an increasing sequence $\{n_k\}_{k\in\NN}$ in $\NN$, and $\e>0$ such that
\[
| \Phi\big( \chi_{n_k}(g)\big)-1| = |\Phi^\ast_g (\chi_{n_k})-1|\geq\e , \quad k\in\NN.
\]
Hence $\Phi^\ast_g (\chi_{n_k})\not\to 1$. Since $\Phi^\ast_g\in \widehat{Z}$ by the claim, we obtain that $\{\chi_n\}_{n\in\NN}$  is not a $\sigma(Z,\widehat{Z})$-null sequence, a contradiction.
\end{proof}


\section{Glicksberg type properties and the property of being a Mackey group} \label{sec:6}


Below we define  some  versions of respecting properties. For a $MAP$ abelian group $G$ and a topological property $\mathcal{P}$, we denote by $\mathcal{P}_{qc}(G)$ the set of all {\em quasi-convex} subsets of $G$ with $\mathcal{P}$. Recall that a locally convex space $E$ has the {\em Grothendieck} property if every weak-$\ast$ convergent sequence in the dual space $E'$ is also weakly convergent, i.e., $\mathcal{S}\big(E',\sigma(E',E)\big)=\mathcal{S}\big(E',\sigma(E',(E'_\beta)')\big)$, where $E'_\beta$ is the strong dual of $E$. Analogously, we say that $E$ has the {\em $\mathcal{P}$-Grothendieck} property if $\mathcal{P}\big(E',\sigma(E',E)\big)=\mathcal{P}\big(E',\sigma(E',(E'_\beta)')\big)$.
\begin{definition} \label{def:qc-Glicksberg} {\em
Let $(G,\tau)$ be a $MAP$ abelian group and $\mathcal{P}$ a topological property. We say that
\begin{enumerate}
\item[{\rm (i)}] $(G,\tau)$  {\em respects}  $\mathcal{P}_{qc}$ if $\mathcal{P}_{qc}(G)=\mathcal{P}_{qc}\big(G,\sigma(G,\widehat{G})\big)$;
\item[{\rm (ii)}]  $(G,\tau)^\wedge$ {\em  respects}  $\mathcal{P}^\ast$ if $\mathcal{P}(G^\wedge)=\mathcal{P}\big(\widehat{G},\sigma(\widehat{G},G)\big)$;
\item[{\rm (iii)}] $(G,\tau)^\wedge$ {\em  respects}  $\mathcal{P}_{qc}^\ast$ if $\mathcal{P}_{qc}(G^\wedge)=\mathcal{P}_{qc}\big(\widehat{G},\sigma(\widehat{G},G)\big)$;
\item[{\rm (iv)}]  $(G,\tau)$  has the {\em   $\mathcal{P}$-Pontryagin--Grothendieck property} if $\mathcal{P}\big(\widehat{G},\sigma(\widehat{G},G)\big)=\mathcal{P}\big(\widehat{G},\sigma(\widehat{G},G^{\wedge\wedge})\big)$.
\end{enumerate} }
\end{definition}

In the case $\mathcal{P}$ is the property $\mathcal{C}$ of being  a compact space and  a $MAP$ abelian group $(G,\tau)$  ($G^\wedge$)  respects $\mathcal{P}_{qc}$ ($\mathcal{P}^\ast$ or $\mathcal{P}_{qc}^\ast$, respectively), we shall say that  the group $G$ ($G^\wedge$) has the {\em $qc$-Glicksberg property} (the {\em weak-$\ast$ Glicksberg property} or the {\em weak-$\ast$ $qc$-Glicksberg property}, respectively). Clearly, if a $MAP$ abelian group $(G,\tau)$ has the Glicksberg property, then it also has the $qc$-Glicksberg property, and   if $(G,\tau)^\wedge$ has the  weak-$\ast$ Glicksberg property, then it has also the  weak-$\ast$  $qc$-Glicksberg property.

\begin{remark} \label{rem:Mackey-dual} {\em
Note that, for a $MAP$ abelian group $G$, if $G^\wedge$ has the weak-$\ast$ Glicksberg property, then it has also the Glicksberg property. But the converse is not true in general. Indeed, let $G$ be a countable dense subgroup of an infinite compact metrizable group $X$. Then $G^\wedge =X^\wedge$ by \cite{Aus,Cha}. Hence  $G^\wedge$ is a  discrete countably infinite group, so  $G^\wedge$ has the Glicksberg property. On the other hand, the group $H:= \big(\widehat{G}, \sigma(\widehat{G},G)\big)$ is a precompact metrizable group. So $H$ contains infinite compact subsets which are not compact in $G^\wedge$.}
\end{remark}


\begin{proposition} \label{p:Grothendieck-dual}
Let $E$ be a real lcs and $\mathcal{P}\in \Pp_0$. If $E$ has the $\mathcal{P}$-Grothendieck property, then $E$ has the $\mathcal{P}$-Pontryagin--Grothendieck property. But the converse is not true in general.
\end{proposition}

\begin{proof}
Let $A\in \mathcal{P}\big(\widehat{E},\sigma(\widehat{E},E)\big)$. Then, by (ii) of Proposition \ref{p:weak*-space-group}, the set $B:=\psi^{-1}(A)$ belongs to $\mathcal{P}\big(E',\sigma(E',E)\big)$. So $B\in \mathcal{P}\big(E',\sigma(E',(E'_\beta)')\big)$. Since the compact-open topology $\tau_k$ on $E'$ is weaker than the strong topology, we obtain that $B\in \mathcal{P}\big(E',\sigma(E',(E',\tau_k)')\big)$. Finally, (v) of Proposition \ref{p:weak*-space-group} implies that $A=\psi(B) \in \mathcal{P}\big(\widehat{E},\sigma(\widehat{E},E^{\wedge\wedge})\big)$.

To prove the last assertion, let $E$ be a separable non-reflexive Banach space. Being a reflexive group \cite{Smith}, $E$ trivially has the $\mathcal{S}$-Pontryagin--Grothendieck property. However, a separable Banach space with the Grothendieck property must be reflexive, so $E$ does not have the Grothendieck property.
\end{proof}

Below we show that the  weak-$\ast$ $qc$-Glicksberg property is dually connected with the property being a Mackey group. Let us recall the definition of Mackey groups.
Two topologies  $\tau$ and $\nu$ on an abelian group $G$  are said to be {\em compatible } if $\widehat{(G,\tau)}=\widehat{(G,\nu)}$.
Being motivated by the classical Mackey--Arens theorem the following notion was  introduced and studied in \cite{CMPT}: a locally quasi-convex abelian group $(G,\tau)$ is called a {\em Mackey group} if for every compatible locally quasi-convex group topology $\nu$ on $G$ it follows that $\nu\leq \tau$. By  Theorem~4.2 of \cite{CMPT}, every $g$-barrelled group is Mackey. The next proposition generalizes this result.

\begin{proposition} \label{p:g-barrelled-Mackey}
Let $(G,\tau)$ be a locally quasi-convex abelian group. If  every $A\in \mathcal{C}_{qc}\big(\widehat{G},\sigma(\widehat{G},G)\big)$ is equicontinuous (for example, $G$ is $g$-barrelled), then $G$ is a Mackey group.
\end{proposition}

\begin{proof}
Let $\nu$ be a locally quasi-convex group topology on $G$ compatible with $\tau$ and let $U$ be a quasi-convex $\nu$-neighborhood of zero. Then $U^{\triangleright}$ is $\sigma(\widehat{G},G)$-compact and quasi-convex by Fact \ref{f:Mackey-Ban}. So $U^{\triangleright}$ is equicontinuous (with respect to the original topology $\tau$).  Fact \ref{f:Mackey-Ban} implies that $U=U^{\triangleright\triangleleft}$ is also a $\tau$-neighborhood of zero. Thus $\nu\leq\tau$ and hence $G$ is a Mackey group.
\end{proof}

Below we obtain another sufficient condition of being a Mackey group.
\begin{proposition} \label{p:Mackey-weak-qc-Glick}
Let $(G,\tau)$ be a locally quasi-convex group such that the canonical homomorphism $\alpha_G$ is continuous. If $(G,\tau)^\wedge$ has the weak-$\ast$ $qc$-Glicksberg property, then  $(G,\tau)$ is a Mackey group. Consequently, if a reflexive abelian group $(G,\tau)$ is such that $(G,\tau)^\wedge$ has the $qc$-Glicksberg property (in particular, the Glicksberg property), then  $(G,\tau)$ is a Mackey group.
\end{proposition}
\begin{proof}
Let $\nu$ be a locally quasi-convex topology on $G$ compatible with $\tau$ and let $U$ be a closed quasi-convex $\nu$-neighborhood of zero. Fact \ref{f:Mackey-Ban} implies that the quasi-convex subset $K:=U^\triangleright$ of $\widehat{G}$ is $\sigma(\widehat{G},G)$-compact, and hence $K$ is a compact subset of  $G^\wedge$ by the weak-$\ast$ $qc$-Glicksberg property. Note that, by definition, $K^\triangleright$ is a neighborhood of zero in $G^{\wedge\wedge}$. As $\alpha_G$ is continuous,  $U=K^\triangleleft =\alpha_G^{-1}(K^\triangleright)$ is a  $\tau$-neighborhood of zero in $G$. Hence $\nu\leq\tau$. Thus $(G,\tau)$ is a Mackey group.

The last assertion follows from the fact that the weak-$\ast$ $qc$-Glicksberg property coincides with the $qc$-Glicksberg property for any reflexive group.
\end{proof}

\begin{remark} \label{rem:Mackey-c0(S)} {\em
In the last assertion of Proposition \ref{p:Mackey-weak-qc-Glick} the reflexivity of $G$ is essential. Indeed, let $G$ be a proper dense subgroup of a compact metrizable abelian group $X$. Then $G^\wedge =X^\wedge$ (see \cite{Aus,Cha}), and hence the discrete group $G^\wedge$ has the Glicksberg property. Denote by $\mathfrak{p}_0$ the product topology on the group $c_0 (\Ss):=\{ (z_n)\in \Ss^\NN: z_n\to 1\}$ induced from $\Ss^\NN$. Then, by \cite[Theorem 1]{Gab}, there is a locally quasi-convex topology  $\mathfrak{u}_0$ on $c_0 (\Ss)$ compatible with $\mathfrak{p}_0$ such that $\mathfrak{p}_0 <\mathfrak{u}_0$. Thus the group $G:=\big( c_0 (\Ss),\mathfrak{p}_0)$ is a precompact arc-connected metrizable group such that $G^\wedge$ has the Glicksberg property, but $G$ is {\em not} a Mackey group. Consequently, $G^\wedge$ does not have the  weak-$\ast$ $qc$-Glicksberg property. }
\end{remark}



Every real barrelled locally convex space  
is a Mackey group by \cite{CMPT} (this also follows from Propositions \ref{p:space-group-barrelled} and \ref{p:g-barrelled-Mackey}). Since every real reflexive locally convex space  $E$ is barrelled by \cite[Proposition 11.4.2]{Jar}, we obtain that $E$ is a Mackey group. This result motivates the following problem.
\begin{problem} \label{q:Mackey-refl}
Characterize reflexive abelian groups which are Mackey groups.
\end{problem}

Not every reflexive group is Mackey, see \cite{Cha-MP}. Moreover, there exists a reflexive group which does not admit a Mackey group topology, see \cite{Aus3,Gabr-A(s)-Mackey}.
However, if a reflexive group $G$ is of finite exponent, it is a Mackey group as the following theorem shows.
\begin{theorem} \label{t:Mackey-finite-exp}
A reflexive abelian group $(G,\tau)$ of finite exponent is a Mackey group.
\end{theorem}

\begin{proof}
Since $(G,\tau)$ is locally quasi-convex, Proposition \ref{p:nuclear-Mackey} implies that  $(G,\tau)^\wedge$ has the Glicksberg property. Thus  $(G,\tau)$ is a Mackey group by Proposition \ref{p:Mackey-weak-qc-Glick}.
\end{proof}

\begin{corollary} \label{c:Mackey-C(X,F)-discrete}
Let $X$ be a zero-dimensional realcompact $k$-space and let $\ff$ be a finite abelian group. Then  $\CC(X,\ff)$ is  a Mackey group.
\end{corollary}
\begin{proof}
The assertion immediately follows from the main result of \cite{PolSmen} and Theorem \ref{t:Mackey-finite-exp}.
\end{proof}

\begin{remark} {\em
Any metrizable and  precompact abelian group of finite exponent is a Mackey group, see \cite[Example 4.4]{BTAVM}. If $G$ is a metrizable reflexive group, then $G$ must be complete by \cite[Corollary 2]{Cha}. So there are non-reflexive Mackey groups of finite exponent. Moreover, since the group $G$ in Example \ref{exa:non-reflexive-g-barrelled} is not $c_0$-barrelled, we obtain that there are Mackey groups which are not $g$-barrelled.}
\end{remark}

Recall that an lcs $E$ has the compact convex property $(ccp)$ if the absolutely convex hull of any compact subset of $E$ is relatively compact in $E$. Analogously we say that a locally quasi-convex abelian group $G$ has a {\em compact quasi-convex property $(cqcp)$} if  the quasi-convex hull of any compact subset of $G$ is relatively compact in $G$. Clearly, every real lcs $E$  with $(ccp)$ has also $(cqcp)$.
\begin{proposition} \label{p:cqcp-property}
Let $G$ be a locally quasi-convex abelian group. If $G$ is $g$-barrelled, then the group $\big( \widehat{G}, \sigma(\widehat{G},G)\big)$ has $(cqcp)$, but the converse is not true in general.
\end{proposition}

\begin{proof}
Let $K$ be a compact subset of $H:=\big( \widehat{G}, \sigma(\widehat{G},G)\big)$. Then $K$ is equicontinuous by the $g$-barrelledness of $G$. Now Fact \ref{f:Mackey-Ban} implies that $K^{\triangleleft\triangleright}$ is $\sigma(\widehat{G},G)$-compact and quasi-convex. Thus $H$ has $(cqcp)$. For the last assertion, see Remark 15 of \cite{CMPT}.
\end{proof}

\begin{remark} {\em
Let $(G,\tau)$ be a $MAP$ abelian group which respects $\mathcal{P}\in\mathfrak{P}$ and let $\nu$ be a group topology on $G$ compatible with $\tau$. If $\tau^+\leq \nu\leq \tau$, then clearly $(G,\nu)$ respects $\mathcal{P}$ as well. But if $\nu >\tau$ it may happen that $(G,\nu)$ does not respect $\mathcal{P}$. Indeed, let $(G,\nu)$ be a real Banach space without the Schur property and let $\tau=\sigma(E,E')$. Then the space $(G,\tau)$ has the Schur property by Proposition \ref{p-Weak=+}. For a more non-trivial example, consider the free lcs $L(\mathbf{s})$ which respects all properties $\mathcal{P}\in\Pp_0$ by Theorem \ref{t:L(X)-Schur}, however the space $\big(L(\mathbf{s}), \mu(L(\mathbf{s}),C(\mathbf{s}))\big)$ does not have the Schur property, see Step 3 of the proof of Theorem 2.4 in \cite{Gabr-L(s)}.}
\end{remark}

%




\bibliographystyle{amsplain}


\end{document}